\documentclass[11pt]{article}
\usepackage{amsfonts}
\usepackage{graphics}
\usepackage{indentfirst}
\usepackage{color}
\usepackage{cite}
\usepackage{latexsym}
\usepackage[paper=a4paper, left=1.6cm, right=1.6cm, top=1.8cm, bottom=1.6cm, headheight=5.5pt, footskip=0.8cm, footnotesep=0.8cm, centering, includefoot]{geometry}
\usepackage{amsmath}
\allowdisplaybreaks
\usepackage{amssymb}
\usepackage[colorlinks, linkcolor=red]{hyperref}
\hypersetup{urlcolor=red, citecolor=red}
\usepackage[dvips]{epsfig}
\usepackage{amscd}

\usepackage{amsthm}
\usepackage{mathrsfs}
\usepackage{verbatim}
\newtheorem{theorem}{Theorem}[section]
\newtheorem{remark}{Remark}[section]
\newtheorem{defn}{Definition}[section]
\newtheorem{lemma}{Lemma}[section]
\newtheorem{corollary}{Corollary}[section]

\allowdisplaybreaks
\let\pa=\partial
\let\f=\frac

\let\D=\Delta
\let\r=\rho
\let\vr=\varrho

\def\Om{\Omega}
\def\na{\nabla}

\def\cb{{\mathcal B}}

\def\D{{\mathcal D}}
\def\E{{\mathcal E}}

\def\cm{{\mathcal M}}

\renewcommand{\div}{{\rm div}}

\newcommand{\lb}{\left(}
\newcommand{\rb}{\right)}

\DeclareMathOperator{\divv}{div}
\DeclareMathOperator{\loc}{loc}
\DeclareMathOperator{\curl}{curl}

\makeatletter
\@addtoreset{equation}{section}
\makeatother
\makeatletter
\@addtoreset{equation}{section}
\makeatother

\title{Global stability for the compressible isentropic magnetohydrodynamic equations in 3D bounded domains with Navier-slip boundary conditions
\thanks{Liu's research was partially supported by Natural Science Foundation of Changchun Normal University (No. CSJJ2024003GZR), Wu's research was partially supported by Fujian Alliance of Mathematics (No. 2023SXLMMS08), and Zhong's research was partially supported by Fundamental Research Funds for the Central Universities (No. SWU--KU24001) and National Natural Science Foundation of China (No. 12371227).}
}

\author{Yang Liu$\,^{\rm 1}\,$,\ Guochun Wu$\,^{\rm 2}\,$,\
	Xin Zhong$\,^{\rm 3}\,$ {\thanks{E-mail addresses:
			liuyang0405@ccsfu.edu.cn (Y. Liu), guochunwu@126.com (G. Wu),  xzhong1014@amss.ac.cn (X. Zhong).}}\date{}\\
	\footnotesize $^{\rm 1}\,$
College of Mathematics, Research Institute for Scientific and Technological Innovation,
\\ \footnotesize Changchun Normal University, Changchun 130032, P. R. China\\
\footnotesize $^{\rm 2}\,$
School of Mathematics and Statistics, Xiamen University of Technology, Xiamen 361024, P. R. China\\
\footnotesize $^{\rm 3}\,$ School of Mathematics and Statistics, Southwest University, Chongqing 400715, P. R. China}
\begin{document}
\maketitle

\begin{abstract}
We study the global stability of large solutions to the compressible isentropic magnetohydrodynamic equations in a three-dimensional (3D) bounded domain with Navier-slip boundary conditions. It is shown that the solutions converge to an equilibrium state exponentially in the $L^2$-norm provided the density is essentially uniform-in-time bounded from above. Moreover, we also obtain that the density and magnetic field converge to their equilibrium states exponentially in the $L^\infty$-norm if additionally the initial density is bounded away from zero. These greatly improve the previous work in (J. Differential Equations 288 (2021), 1--39), where the authors considered the torus case and required the $L^6$-norm of the magnetic field to be uniformly bounded as well as zero initial total momentum and an additional restriction $2\mu>\lambda$ for the viscous coefficients. This paper provides the first global stability result for large strong solutions of compressible magnetohydrodynamic equations in 3D general bounded domains.
\end{abstract}



%\tableofcontents

\section{Introduction}
Let $\Omega\subset\mathbb{R}^3$ be a domain, the evolution of electrically fluids in the presence of magnetic field occupying in $\Omega$ is described by the compressible isentropic magnetohydrodynamic (MHD for short) equations
\begin{align}\label{a1}
\begin{cases}
\rho_t+\divv (\rho u)=0,\\
(\rho u)_t+\divv (\rho u\otimes u)+\nabla P-\mu\Delta u-(\mu+\lambda)\nabla\divv  u=(\nabla\times H)\times H,\\
H_t-\nabla\times(u\times H)=-\nabla\times(\nu\nabla\times H),\\
\divv H=0,
\end{cases}
\end{align}
where the unknowns $\rho$, $u=(u^1, u^2, u^3)$, $H=(H^1, H^2, H^3)$, and $P=a\rho^\gamma\ (a>0,\gamma>1)$ stand for the density, velocity, magnetic field, and pressure, respectively. The constants $\mu$ and $\lambda$ represent the shear viscosity and the bulk viscosity respectively satisfying the physical restrictions
\begin{align*}
\mu>0,\quad 2\mu+3\lambda\ge 0,
\end{align*}
while $\nu>0$ is the resistivity coefficient. The system \eqref{a1} is a combination of the compressible isentropic Navier--Stokes equations of fluid dynamics and Maxwell's equations of electromagnetism. The electric field $E$ can be written in terms of the magnetic field $H$ and the velocity $u$ as follows
\begin{align*}
E=\nu\curl H-u\times H.
\end{align*}
Thus, the electric field $E$ does not appear in $\eqref{a1}_3$ due to Faraday's law.
Without loss of generality, throughout the paper we assume that
$$ a=\nu=1.$$

Let $\Omega\subset\mathbb{R}^3$ be a simply connected smooth bounded domain,
we consider \eqref{a1} with the initial data
\begin{align}\label{1.2}
\rho(x, 0)=\rho_0(x),\ \rho u(x, 0)=\rho_0u_0(x), \ H(x, 0)=H_0(x), \quad x\in \Omega,
\end{align}
and slip boundary conditions
\begin{align}\label{a4}
u\cdot n=0, \ \curl u\times n=0,\,
H\cdot n=0, \ \curl H\times n=0,\quad & x\in\partial\Omega,\ t>0,
\end{align}
where $n$ is the unit outward normal vector to $\pa\Omega$.

Recently there have been numerous studies on the global existence of solutions to the compressible MHD equations \eqref{a1}. Kawashima \cite{SK83} established the global unique smooth solution with the initial data being close to a non-vacuum equilibrium in $H^3(\mathbb{R}^3)$. Based on Lions--Feireisl's framework \cite{L98,F04}, Hu and Wang \cite{HW10} proved the global existence of finite energy weak solutions in a three-dimensional (3D for short) bounded domain with Dirichlet boundary conditions for the velocity and magnetic field. With the help of the {\it effective viscous flux}, Suen--Hoff \cite{SH12} and Liu--Yu--Zhang \cite{LYZ13} obtained the global existence of weak solutions with the initial density and gradients of the initial velocity/magnetic field being bounded in $L^2$ and $L^\infty$, respectively. Meanwhile, Li--Xu--Zhang \cite{LXZ13} showed the global existence of classical solutions to the 3D Cauchy problem with small initial energy but possibly large oscillations and vacuum states, which generalized the corresponding result in \cite{HLX12} for the compressible Navier--Stokes system. Later on, Hong--Hou--Peng--Zhu \cite{HHPZ} obtained global smooth solutions by requiring $\big[(\gamma-1)^{\frac19}+\nu^{-\frac14}\big]E_0$ to be suitably small, where $E_0$ is the initial energy. More recently, global strong solutions to the initial-boundary value problem with Navier-slip boundary conditions were investigated in \cite{CHPS23,CHS23}. Moreover, there are some interesting results on compressible non-resistive MHD equations, please refer to \cite{DWZ24,JJ19,WZ23,WZ22} and references therein.

Some important progress has also been made regarding the global stability and long-time behavior of solutions to the 3D compressible isentropic MHD equations. Among them, Wu--Zhang--Zou \cite{WZZ21} obtained the optimal time-decay rates of weak solutions with discontinuous initial data, i.e., for $t\geq1$,
\begin{align*}
\begin{cases}
\|(\rho-\tilde{\rho},u,H)(\cdot,t)\|_{L^r}\leq Ct^{-\frac32(1-\frac1r)},\ 2\leq r\leq\infty,\\
\|(\nabla u,\nabla H)(\cdot,t)\|_{L^2}\leq Ct^{-\frac54}, \\
\|(\nabla^2H,H_t)(\cdot,t)\|_{L^2}\leq Ct^{-\frac74},
\end{cases}
\end{align*}
where $\tilde{\rho}$ is the constant far-field density.
For sufficiently smooth and small perturbations of a stable reference state, Zhang and Zhao \cite{ZZ10} established time decay estimates of global classical solutions to the 3D Cauchy problem under an additional $L^p$ assumption. Li and Yu \cite{LY11} derived the optimal temporal decay estimates on the strong solutions with the initial perturbation belonging to $L^1$-spaces. The key issue of the method depends heavily on the linearized equations and the initial data being small perturbations of a equilibrium.
At the same time, several authors dealt with the long-time behavior of solutions when the initial energy is suitably small. It was shown in \cite{LXZ13} that, for any $p>2$, the unique classical solution to the 3D Cauchy problem admits the long-time
behavior
\begin{align}\label{z8}
\lim_{t\rightarrow\infty}\big(\|\rho(\cdot,t)-\tilde{\rho}\|_{L^p}
+\|\nabla u\|_{L^2}+\|\nabla H\|_{L^2}\big)=0.
\end{align}
If additional $\rho_0\in L^1(\mathbb{R}^3)$, L{\"u}--Shi--Xu \cite{LSX16} obtained that,
for any $t\geq1$,
\begin{align*}
\begin{cases}
\|\nabla H(\cdot,t)\|_{L^p}\leq Ct^{-1+\frac{6-p}{4p}},\ &p\in[2,6],\\
\|\nabla u(\cdot,t)\|_{L^p}\leq Ct^{-1+\frac1p},\ &p\in[2,6], \\
\|P(\cdot,t)\|_{L^r}\leq C(r)t^{-1+\frac1r},\ &r\in(1,\infty).
\end{cases}
\end{align*}
When $\Omega$ is a bounded domain and $\bar\rho_0=\frac{1}{|\Omega|}\int_\Omega \rho_0dx$, Chen--Huang--Peng--Shi \cite{CHPS23} derived that there exist positive constants $C$ and $\eta_0$ independent of $t$ such that, for any $t>0$, $r\in[1,\infty)$, and $p\in[1,6]$,
\begin{align}\label{z7}
\|\rho(\cdot,t)-\bar\rho_0\|_{L^r}+\|u(\cdot,t)\|_{W^{1,p}}
+\|\sqrt{\rho}\dot{u}(\cdot,t)\|_{L^2}^2+\|H(\cdot,t)\|_{H^2}\le Ce^{-\eta_0t},
\end{align}
while Chen--Huang--Shi \cite{CHS23} proved that \eqref{z8} also holds for 3D exterior domains. Apart from small perturbations and small initial energy, there is another category concerning the long-time behavior of solutions which can be simply summarized as {\it boundedness implies time decay}. For example, Chen--Huang--Xu \cite{CHX19} demonstrated the global stability and the time-decay rates of large solutions in $\mathbb{R}^3$ provided the density and the magnetic field are bounded from above uniformly in time in the H{\"o}lder space $C^\alpha$ with $\alpha$ being sufficiently small and in $L^\infty$ respectively. Zhu and Zi \cite{ZZ21} proved that all derivatives of the smooth solution in $\mathbb{T}^3$ decay exponentially fast to the equilibrium when the density is uniformly bounded from up and below and the $L^6$ norm of the magnetic field is uniformly bounded. However, up to now, the global dynamics of strong solutions to \eqref{a1} with {\it bounded density} in 3D general bounded domains is still unknown. In fact, this is the primary objective of this paper.

Before stating our main result precisely, we describe the notation throughout.
The symbol $\Box$ denotes the end of a proof and $A\triangleq B$ means $A=B$ by definition.
We denote by $\dot{f}\triangleq f_t+u\cdot\nabla f$ the material derivative of $f$ and
\begin{align*}
\int fdx\triangleq\int_\Omega fdx, \quad\bar f\triangleq\frac{1}{|\Omega|}\int_\Omega fdx.
\end{align*}
For $1\le p\le \infty$ and $k\ge 1$, the standard Sobolev spaces are defined as follows
\begin{gather*}
L^p=L^p(\Omega), \ W^{k, p}=W^{k, p}(\Omega), \ H^k=W^{k, 2},\\
\tilde H^1 \triangleq\left\{v\in H^{1}(\Omega)| (v\cdot n)|_{\partial\Omega}=(\curl v\times n)|_{\partial\Omega}=0\right\}.
\end{gather*}
Moreover, we write the initial total energy of \eqref{a1} as
\begin{align}
C_0\triangleq\int\Big(\f12\rho_0|u_0|^2+\f12|H_0|^2+G(\rho_0)\Big)dx,
\end{align}
with
\begin{align}\label{a7}
G(\rho)\triangleq\rho\int_{\bar\rho}^\rho\f{P(s)-P(\bar\rho)}{s^2}ds.
\end{align}
Now we define precisely what we mean by strong solutions to \eqref{a1}--\eqref{a4}.
\begin{defn}[Strong solutions]
For $T>0$, $(\rho, u, H)$ is called a strong solution to the problem \eqref{a1}--\eqref{a4} on $\Omega\times [0, T]$ if for some $q\in (3, 6]$,
\begin{align}\label{a8}
\begin{cases}
0\le \rho\in C([0, T]; W^{1, q}),\ \rho_t\in C([0, T]; L^{q}),\\
(u, H)\in C([0, T]; H^2)\cap L^2(0, T; W^{2, q}),\\
(\sqrt{\rho}u_t, H_t)\in L^\infty(0, T; L^2), \ (u_t, H_t)\in L^2(0, T; H^1),
\end{cases}
\end{align}
and $(\rho, u, H)$ satisfies \eqref{a1} a.e. in $\Omega\times[0, T]$.
\end{defn}

Our main result can be stated as follows.
\begin{theorem}\label{thm1}
Assume that the initial data $(\rho_0\ge 0, u_0, H_0)$ satisfies
\begin{align}
K\triangleq\|\rho_0-\bar\rho_0\|_{L^2}
+\|\sqrt{\rho_0}u_0\|_{L^2}+\|\nabla u_0\|_{H^1}+\|\nabla H_0\|_{H^1}<\infty.
\end{align}
Let $(\rho, u, H)$ be a global strong solution to the problem \eqref{a1}--\eqref{a4} verifying that
\begin{align}\label{a9}
\sup_{t\ge 0}\|\rho(\cdot, t)\|_{L^\infty}\le \hat\rho
\end{align}
for some positive constant $\hat\rho$. Then there exist two positive constants $C_1$ and $\eta_1$,
which are dependent on $\hat\rho$ and $K$, but independent of $t$, such that
\begin{align}\label{1.8}
\|(\rho-\bar\rho_0,\sqrt{\rho}u, H, \nabla u,\nabla H,\sqrt{\rho}\dot{u}, H_t)(\cdot, t)\|_{L^2}\le C_1e^{-\eta_1t}.
\end{align}
If additionally $\inf\limits_{x\in\Omega}\rho_0(x)\ge \rho_*>0$, then
there exist positive constants $C_2$ and $\eta_2$,
which are dependent on $\rho_*$, $\hat\rho$ and $K$, but independent of $t$, such that
\begin{align}\label{b1.11}
\|(\rho-\bar\rho_0, H)\|_{L^\infty}\le C_2e^{-\eta_2t}.
\end{align}
\end{theorem}

\begin{remark}
It should be pointed out that global strong solutions to the problem \eqref{a1}--\eqref{a4} with vacuum have been proven in \cite{CHPS23}.
One of the key ingredients in \cite{CHPS23} is to derive a time-independent
upper bound of the density. Hence, via similar arguments as those in \cite{CHPS23}, one can show that a strong solution $(\rho,u,H)$ satisfying \eqref{a9} indeed exists provided that the initial energy is suitably small.
\end{remark}

\begin{remark}
Compared with \cite{ZZ21}, there are several innovations. First of all,
we deal with the general bounded domains rather than the torus $\mathbb{T}^3$.
Secondly, we remove the technical restriction $2\mu>\lambda$. Thirdly, there is no need to impose the $L^6$-norm of the magnetic field to be uniformly bounded. Fourthly, we do not require the assumption of zero initial total momentum. Last but not least, the density is assumed to be bounded from below initially in order to derive the exponential decay rate \eqref{b1.11}.
\end{remark}

\begin{remark}
The stability result of the density in \eqref{b1.11} can be regarded as a supplement of \eqref{z7}.
\end{remark}

As a direct application of Theorem \ref{thm1}, we have the following result (cf. \cite[Theorem 1.2]{LWZ25}), which shows that the vacuum state will persist for any time as long as the initial
density contains vacuum.
\begin{corollary}\label{thm2}
Assume that all conditions of Theorem \ref{thm1} are satisfied. If additionally $\inf_{x\in\Omega}\rho_0(x)=0$, then it holds
that, for any $t\ge 0$,
\begin{align*}
\inf_{x\in\Omega}\rho(x, t)=0.
\end{align*}
\end{corollary}

Now we sketch the main ideas used in the proof of Theorem \ref{thm1}. First of all, we prove the exponential stability of $\|(\rho-\bar\rho_0, \sqrt\rho u, H)\|_{L^2}$. On one hand, employing the basic energy estimate and the Bogovskii operator (see Lemma \ref{l22}) to the problem \eqref{a1}--\eqref{a4}, one can deduce an energy-dissipation inequality
\begin{align}\label{z1}
\f{d}{dt}\widetilde{\E}(t)+\widetilde{\D}(t)\le 0,
\end{align}
where $\widetilde{\E}(t)$ and $\widetilde{\D}(t)$ are equivalent to $\|(\rho-\bar\rho_0, \sqrt\rho u, H)\|_{L^2}^2$ and $\|(\rho-\bar\rho_0, \na u, \na H)\|_{L^2}$, respectively. On the other hand, by making use of the $L^p$-estimate of the div-curl system (see \eqref{f6}), Poincar\'e's inequality (see \eqref{z2.5}), and \eqref{a9}, one has that
\begin{align*}
\|\na u\|_{L^p}\le C\big(\|\divv u\|_{L^p}+\|\curl u\|_{L^p}\big),\ \
\|\sqrt\rho u\|_{L^2}\le C\|\nabla u\|_{L^2},\ \
\|H\|_{L^2}\le C\|\na H\|_{L^2},
\end{align*}
which implies that $\widetilde{\D}(t)\ge C\widetilde{\E}(t)$. This combined with \eqref{z1} yields the desired exponential decay of $\|(\rho-\bar\rho_0, \sqrt\rho u, H)\|_{L^2}$.
Next, we show the exponential stability of $\|(\na u,\na H)(\cdot,t)\|_{L^2}$. To this end, taking advantage of good properties of the {\it effective viscous flux} $F$ and the {\it vorticity} $\curl u$ as well as delicate energy estimates, we can derive an important inequality \eqref{w30}. Note that $\|(\na u, \na H)(\cdot, t)\|_{L^2}$ is sufficiently small for large enough $t$ (see \eqref{b36}), this together with \eqref{w15} gives a crucial Lyapunov-type energy inequality \eqref{b37}, which combined with Gronwall's inequality indicates the desired exponential decay of $\|(\nabla u, \na H)(\cdot,t)\|_{L^2}$.
Furthermore, in order to obtain the exponential stability of $\|(\sqrt{\rho}\dot{u}, H_t)(\cdot,t)\|_{L^2}$, we need to improve the uniform-in-time bounds of $\|(\sqrt{\rho}\dot{u}, H_t)(\cdot,t)\|_{L^2}$. In this process, we have to deal with some boundary integrals, such as $\int_{\pa\Omega}F_t(\dot u\cdot n)dS$ and so on. To overcome this difficulty, inspired by \cite{CL23}, it follows from $(u\cdot n)|_{\partial\Omega}=0$ that
\begin{align}\label{a15}
u\cdot\na u\cdot n=-u\cdot\na n\cdot u,\quad \text{on}~\partial\Omega,
\end{align}
which indicates that
\begin{align*}
(\dot u+(u\cdot\na n)\times u^\bot)\cdot n=0, \quad \text{on}~\partial\Omega,
\end{align*}
with $u^\bot\triangleq-u\times n$, and the key point is to control $\int_{\pa\Om}(u\cdot\na n\cdot u)FdS$ (see \eqref{w43}).
Combining these key facts with lower-order energy estimates, we obtain a Lyapunov-type energy inequality \eqref{w58}. This will lead to the exponential decay of $\|(\sqrt{\rho}\dot{u}, H_t)(\cdot,t)\|_{L^2}$ immediately.
Finally, we establish that $\|\rho-\bar\rho_0\|_{L^\infty}$ decays exponentially. The key ingredient is to get a time-independent positive lower bound of the density $\rho$. This is the case by modifying the methods used in \cite{WZ24,LWZ25}. Then we can prove the desired by the damping mechanism of density (see Lemma \ref{l34}).

The rest of this paper is organized as follows. Some important inequalities and auxiliary lemmas will be given in Section \ref{sec2}, while Section \ref{sec3} is devoted to proving Theorem \ref{thm1}.

\section{Preliminaries}\label{sec2}

In this section we recall some known facts and elementary inequalities which will be used later. We start with the following Gagliardo--Nirenberg inequality (cf. \cite[Lemma 2.3]{LWZ25}).
\begin{lemma}\label{l21}
(Gagliardo--Nirenberg inequality, special case).
Assume that $\Omega$ is a bounded Lipschitz domain in $\mathbb{R}^3$. For $p\in [2, 6]$, $q\in (1, \infty)$, and $r\in (3, \infty)$, there exist generic constants $C_i>0\ (i\in\{1,2,3,4\})$ which may depend only on $p$, $q$, $r$, and $\Omega$ such that, for any $f\in H^1$ and $g\in L^q\cap D^{1, r}$,
\begin{gather}
\|f\|_{L^p}\le C_1\|f\|_{L^2}^\f{6-p}{2p}\|\na f\|_{L^2}^\f{3p-6}{2p}+C_2\|f\|_{L^2},\label{f1}\\
\|g\|_{L^\infty}\le C_3\|g\|_{L^q}^{\frac{q(r-3)}{3r+q(r-3)}}\|\na g\|_{L^r}^{\frac{3r}{3r+q(r-3)}}+C_4\|g\|_{L^2}.\label{f2}
\end{gather}
Moreover, if $\int_{\Omega}f(x)dx=0$ or $(f\cdot n)|_{\partial\Omega}=0$ or $(f\times n)|_{\partial\Omega}=0$, we can choose $C_2=0$. Similarly, the constant $C_4=0$ provided $\int_{\Omega}g(x)dx=0$ or $(g\cdot n)|_{\partial\Omega}=0$ or $(g\times n)|_{\partial\Omega}=0$.
\end{lemma}

Next, we introduce the Bogovskii operator in a bounded domain, which plays an important role in controlling $\|P\|_{L^2}^2$. One has the following conclusion (cf. \cite[Lemma 3.17]{NS04}).
\begin{lemma}\label{l22}
Let $\Omega\subset\mathbb{R}^3$ be a bounded Lipschitz domain. Then, for any $p\in(1,\infty)$, there exists a linear operator $\mathcal{B}=\left[\mathcal{B}_{1}, \mathcal{B}_{2}, \mathcal{B}_{3}\right]: L^{p}(\Omega) \rightarrow \big(W_{0}^{1, p}(\Omega)\big)^{3}$ such that
\begin{align*}
\begin{cases}
\operatorname{div} \mathcal{B}[f]=f, & x \in \Omega, \\
\mathcal{B}[f]=0, & x \in \partial \Omega,
\end{cases}
\end{align*}
and
\begin{align*}
\|\nabla\mathcal{B}[f]\|_{L^{p}} \leq C(p,\Omega)\|f\|_{L^{p}}.
\end{align*}
Moreover, if $f=\operatorname{div} g$ with $g\in L^p(\Omega)$ satisfying $(g \cdot n)|_{\partial \Omega}=0$, it holds that
\begin{align*}
\|\mathcal{B}[f]\|_{L^{p}} \leq C(p,\Omega)\|g\|_{L^{p}}.
\end{align*}
\end{lemma}

Since $(u, H)$ does not vanish on the boundary,
we need to use the following generalized Poincar{\'e} inequality (cf. \cite[Lemma 8]{BS2012}).
\begin{lemma}\label{l23}
Let $\Omega\subset\mathbb{R}^3$ be a bounded Lipschitz domain. Then, for $1<p<\infty$, there exists a positive constant $C$ depending only on $p$ and $\Omega$ such that
\begin{align}\label{z2.5}
\|f\|_{L^p} \leq  C\|\nabla f\|_{L^p},
\end{align}
for each vector field $f\in W^{1,p}(\Omega)$ satisfying either $(f\cdot n)|_{\partial\Omega}=0$ or $(f\times n)|_{\partial\Omega}=0$.
\end{lemma}

Next, we set
\begin{align}\label{f5}
F\triangleq(2\mu+\lambda)\divv u-(P-\bar P)-\f12|H|^2, \quad \omega\triangleq\na\times u,
\end{align}
where $F$ and $\omega$ denotes the {\it effective viscous flux} and the {\it vorticity}, respectively. For $F$, $\omega$, $\na u$, and $\na H$, we have the following key {\it a priori} estimates (cf. \cite[Lemma 2.2]{CHPS23}).
\begin{lemma}\label{l24}
Let $(\rho, u, H)$ be a classical solution of \eqref{a1}--\eqref{a4} on $\Omega\times(0, T]$. Then, for any $p\in[2, 6]$ and $1<q<+\infty$, there exists a positive constant $C$ depending only on $p$, $q$, $\mu$, $\lambda$, and $\Omega$ such that
\begin{align}\label{f6}
\|\nabla u\|_{L^q}&\le C\big(\|\divv u\|_{L^q}+\|\omega\|_{L^q}\big),
\\ \label{f7}
\|\nabla H\|_{L^q}&\le C\|\curl H\|_{L^q},
\\ \label{f8}
\|\nabla F\|_{L^p}&\le C\big(\|\rho\dot u\|_{L^p}+\|H\cdot\nabla H\|_{L^p}\big),\\
\|\nabla\omega\|_{L^p}&\le C\big(\|\rho\dot u\|_{L^p}+\|H\cdot\nabla H\|_{L^p}+\|\nabla u\|_{L^2}\big),\\
\|F\|_{L^p}&\le C\big(\|\rho\dot u\|_{L^2}+\|H\cdot\nabla H\|_{L^2}\big)^{\frac{3p-6}{2p}}\big(\|\nabla u\|_{L^2}+\|P-\bar P\|_{L^2}+\|H\|_{L^4}^2\big)^{\frac{6-p}{2p}}\nonumber\\
&\quad+C\big(\|\nabla u\|_{L^2}+\|P-\bar P\|_{L^2}+\|H\|_{L^4}^2\big),
\\
\|\omega\|_{L^p}&\le C\big(\|\rho\dot u\|_{L^2}+\|H\cdot\nabla H\|_{L^2}\big)^{\frac{3p-6}{2p}}\|\nabla u\|_{L^2}^{\frac{6-p}{2p}}+C\|\nabla u\|_{L^2}.
\end{align}
Moreover, one has that
\begin{align}\label{f12}
\|\nabla u\|_{L^p}&\le C\big(\|\rho\dot u\|_{L^2}+\|H\cdot\nabla H\|_{L^2}\big)^{\frac{3p-6}{2p}}\big(\|\nabla u\|_{L^2}+\|P-\bar P\|_{L^2}+\|H\|_{L^4}^2\big)^{\frac{6-p}{2p}}\nonumber\\
&\quad+C\big(\|\nabla u\|_{L^2}+\|P-\bar P\|_{L^p}+\||H|^2\|_{L^p}\big),\\
\quad\|\na^2H\|_{L^p}&\le C\|\curl H\|_{W^{1, p}}\le C(\|\curl^2H\|_{L^p}+\|\curl H\|_{L^p}).\label{f13}
\end{align}
\end{lemma}

Finally, we have the following estimates on the material derivative of $u$ (cf. \cite[Lemma 2.10]{CL23}).
\begin{lemma}
Under the assumption of Lemma \ref{l23}, there exists a positive constant $\Lambda$ depending only on $\Omega$ such that
\begin{gather}
\|\dot u\|_{L^6}\le\Lambda\big(\|\nabla\dot u\|_{L^2}+\|\nabla u\|_{L^2}^2),\\
\|\nabla\dot u\|_{L^2}\le\Lambda\big(\|\divv\dot u\|_{L^2}+\|\curl\dot u\|_{L^2}+\|\nabla u\|_{L^4}^2\big).
\end{gather}
\end{lemma}

\section{Proof of Theorem \ref{thm1}}\label{sec3}

This section is devoted to proving Theorem \ref{thm1}.
%In order to obtain the {\it a priori} estimates in what follows, we will derive some energy estimates. Then, Theorem \ref{thm1} is an easy consequence of Lemmas \ref{l31}--\ref{l34}.

\begin{lemma}\label{l31}
Under the assumptions of Theorem \ref{thm1}, there exist two positive constants $C_3$ and $\eta_3$, which are dependent on $K$,
but independent of $t$, such that for any $t\ge 0$,
\begin{align}\label{w1}
\|\rho(\cdot, t)-\bar\rho_0\|_{L^2}+\|\sqrt\rho u(\cdot, t)\|_{L^2}+\|H(\cdot, t)\|_{L^2}\le C_3e^{-\eta_3t}.
\end{align}
\end{lemma}
\begin{proof}[Proof]
Noticing that
\begin{gather}
\bar\rho=\f{1}{|\Omega|}\int\rho(x, t)dx\equiv\f{1}{|\Omega|}\int\rho_0dx=\bar\rho_0,\label{w2}\\
-\Delta u=-\nabla\divv u+\curl\curl u,\ -\nabla\times(\nabla\times H)=\Delta H,\label{z2} \\
(\nabla\times H)\times H=H\cdot\nabla H-\frac12\nabla|H|^2,\ \
\nabla\times(u\times H)=H\cdot\nabla u-u\cdot\nabla H-H\divv u,\label{z3}
\end{gather}
then multiplying $\eqref{a1}_1$, $\eqref{a1}_2$, and $\eqref{a1}_3$ by $G'(\rho)$, $u$, and $H$, respectively, and integration by parts, one gets that
\begin{align}\label{w6}
&\frac{d}{dt}\int\Big(\frac12\rho|u|^2+G(\rho)\Big)dx+(2\mu+\lambda)\int(\divv u)^2dx+\mu\int|\curl u|^2dx\nonumber\\
&=-\int\div (\rho u)G'(\rho)dx-\int u\cdot\nabla(P-\bar P)dx+\int H\cdot\nabla H\cdot udx-\frac12\int u\cdot\nabla|H|^2dx\nonumber\\
&=\int \rho u\cdot\nabla Q(\rho)dx-\int u\cdot\nabla Pdx+\int u^i\partial_j(H^iH^j)dx-\frac12\int u^i\partial_i|H|^2dx\nonumber\\
&=-\int H\cdot \nabla u\cdot Hdx+\frac12\int\divv u|H|^2dx,
\end{align}
and
\begin{align}\label{w7}
\f12\frac{d}{dt}\int|H|^2dx+\int|\nabla H|^2dx=\int H\cdot\nabla u\cdot Hdx-\frac12\int\divv u|H|^2dx.
\end{align}
%Here $G'(\rho)=Q(\rho)-Q(\bar\rho)$, $Q'(\rho)=\frac{P'(\rho)}{\rho}$.
Here and in what follows, we use the Einstein convention that the repeated indices denote the summation.
Hence, we derive from \eqref{w6} and \eqref{w7} that
\begin{align}\label{w8}
\frac{d}{dt}\int\Big(\frac12\rho|u|^2+G(\rho)+\f12|H|^2\Big)dx
+(2\mu+\lambda)\int(\divv u)^2dx+\mu\int|\curl u|^2dx
+\int|\nabla H|^2dx=0.
\end{align}

By \eqref{a7} and \eqref{w2}, there exists a positive constant $\tilde C<1$ depending only on $\gamma$, $\bar\rho_0$, and $\hat\rho$ such that, for any $\rho\in [0, 2\hat\rho]$,
\begin{gather}
\tilde C^2(\rho-\bar\rho_0)^2\le \tilde CG(\rho)\le (P-P(\bar\rho_0))(\rho-\bar\rho_0),\label{w9}\\
\|P-\bar P\|_{L^2}^2\le C\|P-P(\bar\rho_0)\|_{L^2}^2\le C\int G(\rho)dx.\label{w10}
\end{gather}
Then, multiplying $\eqref{a1}_2$ by $\cb[\rho-\bar\rho_0]$, we deduce from \eqref{a8} and Lemma \ref{l21} that
\begin{align*}
\int(P-P(\bar\rho_0))(\rho-\bar\rho_0)dx
&=\frac{d}{dt}\int\rho u\cdot\cb[\rho-\bar\rho_0]dx-\int\rho u\cdot\nabla\cb[\rho-\bar\rho_0]\cdot udx
-\int\rho u\cdot\cb[\rho_t]dx\nonumber\\
&\quad+\mu\int\nabla u\cdot\nabla\cb[\rho-\bar\rho_0]dx+(\mu+\lambda)\int(\rho-\bar\rho_0)\divv udx\nonumber\\
&\quad+\int\Big(\frac12\nabla|H|^2-\divv(H\otimes H)\Big)\cdot\cb[\rho-\bar\rho_0]dx\nonumber\\
&=\frac{d}{dt}\int\rho u\cdot\cb[\rho-\bar\rho_0]dx
+C\|\sqrt\rho u\|_{L^4}^2\|\rho-\bar\rho_0\|_{L^2}+C\|\rho u\|_{L^2}^2\nonumber\\
&\quad+C\|\rho-\bar\rho_0\|_{L^2}\|\nabla u\|_{L^2}+C\|H\|_{L^4}^2\|\rho-\bar\rho_0\|_{L^2}\nonumber\\
&\le \frac{d}{dt}\int\rho u\cdot\cb[\rho-\bar\rho_0]dx
+C\|\sqrt\rho u\|_{L^2}^\frac12\|\sqrt\rho u\|_{L^6}^\frac32\|\rho-\bar\rho_0\|_{L^2}+C\|\nabla u\|_{L^2}^2\nonumber\\
&\quad+C\|\rho-\bar\rho_0\|_{L^2}\|\nabla u\|_{L^2}
+C\|H\|_{L^3}\|H\|_{L^6}\|\rho-\bar\rho_0\|_{L^2}\nonumber\\
&\le \frac{d}{dt}\int\rho u\cdot\cb[\rho-\bar\rho_0]dx+\frac{\tilde C^2}{2}\|\rho-\bar\rho_0\|_{L^2}^2+C\|\nabla u\|_{L^2}^2
+C\|\nabla H\|_{L^2}^2,
\end{align*}
which along with \eqref{w9} and \eqref{w10} leads to
\begin{align*}
\tilde C^2\|\rho-\bar\rho_0\|_{L^2}^2&\le \tilde C\int G(\rho)dx\le \int(P-P(\bar\rho_0))(\rho-\bar\rho_0)dx\nonumber\\
&\le \frac{d}{dt}\int\rho u\cdot\cb[\rho-\bar\rho_0]dx+\frac{\tilde C^2}{2}\|\rho-\bar\rho_0\|_{L^2}^2+C\|\nabla u\|_{L^2}^2
+C\|\nabla H\|_{L^2}^2.
\end{align*}
Thus, by \eqref{f6}, we arrive at
\begin{align}\label{w13}
-\frac{d}{dt}\int\rho u\cdot\cb[\rho-\bar\rho_0]dx+\frac{\tilde C^2}{2}\|\rho-\bar\rho_0\|_{L^2}^2
\le C\|\divv u\|_{L^2}^2+C\|\curl u\|_{L^2}^2+C\|\nabla H\|_{L^2}^2.
\end{align}

Now we choose a positive constant $D_1$ suitably large and define the temporal energy functional
\begin{align*}
\cm_1(t)=D_1\int\Big(\frac12\rho|u|^2+G(\rho)+\frac12|H|^2\Big)dx-\int\rho u\cdot\cb[\rho-\bar\rho_0]dx.
\end{align*}
It follows from \eqref{w9} that
\begin{align*}
\left|\int\rho u\cdot\cb[\rho-\bar\rho_0]dx\right|\le \tilde C_2\int\Big(\frac12\rho|u|^2+G(\rho)\Big)dx.
\end{align*}
Thus, $\cm_1(t)$ is equivalent to $\|(\rho-\bar\rho_0, \sqrt{\rho}u, H)\|_{L^2}^2$ provided we choose $D_1$ large enough. Taking a linear combination of \eqref{w8} and \eqref{w13}, we deduce that, for any $t\ge 0$,
\begin{align}\label{w15}
\frac{d}{dt}\cm_1(t)+\frac{\cm_1(t)}{D_1}+\frac{\|\divv u\|_{L^2}^2+\|\curl u\|_{L^2}^2+\|\nabla H\|_{L^2}^2}{D_1}\le 0.
\end{align}
Integrating \eqref{w15} with respect to $t$ over $[0, \infty)$ gives \eqref{w1}.
\end{proof}

Next, we establish the time-decay rate of $\|\nabla u\|_{L^2}$ and $\|\nabla H\|_{L^2}$.
\begin{lemma}
Under the assumptions of Theorem \ref{thm1}, there exist two positive constant $C_4$ and $\eta_4$, which are dependent on $M$ and $K$, but independent of $t$, such that for any $t\ge 0$,
\begin{align}\label{w17}
\|\nabla u(\cdot, t)\|_{L^2}+\|\nabla H(\cdot, t)\|_{L^2}+\|H(\cdot, t)\|_{L^4}\le C_4e^{-\eta_4t}.
\end{align}
\end{lemma}
\begin{proof}[Proof]
By the definition of the material derivative, \eqref{z2}, and \eqref{z3}, one can rewrite $\eqref{a1}_2$ as
\begin{align}\label{w18}
\rho\dot u+\nabla(P-\bar P)=(2\mu+\lambda)\nabla\divv u-\mu\curl\curl u+H\cdot\nabla H-\frac12\nabla|H|^2.
\end{align}
Multiplying \eqref{w18} by $\dot u$ and integration by parts, we get that
\begin{align}\label{w19}
\int\rho|\dot u|^2dx&=-\int\dot u\cdot\nabla(P-\bar P)dx+(2\mu+\lambda)\int\nabla\divv u\cdot\dot udx-\mu\int\curl\curl u\cdot\dot udx\nonumber\\
&\quad+\int\Big(H\cdot\nabla H-\frac12\nabla|H|^2\Big)\cdot\dot udx
\triangleq\sum_{i=1}^4I_i.
\end{align}

By $\eqref{a1}_1$ and $P=\rho^\gamma$, we have
\begin{align*}
P_t+\div(Pu)+(\gamma-1)P\divv u=0,
\end{align*}
which gives that
\begin{equation}\label{z3.26}
(P-\Bar{P})_t+u\cdot\nabla(P-\Bar{P})+\gamma P\divv u-(\gamma-1)\overline{P\divv u}=0.
\end{equation}
Owing to
\begin{align*}
\overline{P\divv u}=\frac{1}{|\Omega|}\int \rho^\gamma\divv u\mathrm{d}dx\leq C(\hat{\rho},\gamma,\Omega)\left\|\divv u\right\|_{L^2},
\end{align*}
then it follows from \eqref{z3.26} that
\begin{align}\label{z4}
I_1&=-\int u_t\cdot\nabla(P-\bar P)dx-\int u\cdot\nabla u\cdot\nabla(P-\bar P)dx\nonumber\\
&=\frac{d}{dt}\int(P-\bar P)\divv udx-\int\divv u(P-\bar P)_tdx-\int u\cdot\nabla u\cdot\nabla (P-\bar P)dx\nonumber\\
&=\frac{d}{dt}\int(P-\bar P)\divv udx-\gamma\int P(\divv u)^2dx-\int u\cdot\na(P-\bar P)\divv udx\nonumber\\
&\quad+(\gamma-1)\int\divv u\overline{P\divv u}dx-\int u\cdot\nabla u\cdot\nabla Pdx\nonumber\\
&=\frac{d}{dt}\int(P-\bar P)\divv udx-\gamma\int P(\divv u)^2dx+\int (P-\bar P)\nabla u:\nabla udx\nonumber\\
&\quad+(\gamma-1)\int\divv u\overline{P\divv u}dx-\int_{\partial\Omega}Pu\cdot\nabla u\cdot ndS\nonumber\\
&\le \frac{d}{dt}\int(P-\bar P)\divv udx+C\|\nabla u\|_{L^2}^2,
\end{align}
where we have used
\begin{align*}
&\int u\cdot\na(P-\bar P)\divv udx-\int u\cdot\nabla u\cdot\nabla (P-\bar P)dx \notag \\
&=\int u^i\pa_i(P-\bar P)\pa_ju^jdx-\int u^i\partial_iu^j\partial_j(P-\bar P)dx\nonumber\\
&=-\int\pa_iu^i\pa_ju^j(P-\bar P)dx-\int(P-\bar P)\pa_j\pa_iu^ju^idx-\int u^i\partial_iu^j\partial_j(P-\bar P)dx\nonumber\\
&=-\int (P-\bar P)(\divv u)^2dx+\int (P-\bar P)\partial_iu^j\partial_ju^idx-\int_{\partial\Omega}(P-\bar P)u^i\partial_iu^jn^jdS\nonumber\\
&=-\int (P-\bar P)(\divv u)^2dx+\int (P-\bar P)\nabla u:\nabla udx+\int_{\partial\Omega}(P-\bar P)u\cdot\nabla n\cdot udS  \\
&\le C\int |P-\bar P||\nabla u|^2dx+C\int_{\partial\Omega}|u|^2dS
\le C\|\nabla u\|_{L^2}^2,
\end{align*}
due to \eqref{a4}, \eqref{a15}, the trace theorem, \eqref{z2.5}, and \eqref{a9}.

By virtue of \eqref{a4} and \eqref{a15}, we infer from integration by parts that
\begin{align}\label{z5}
I_2&=(2\mu+\lambda)\int_{\partial\Omega}{\rm div}\,u(\dot{u}\cdot n)dS-(2\mu+\lambda)\int{\rm div}\,u{\rm div}\,\dot{u}dx\nonumber\\
&=(2\mu+\lambda)\int_{\partial\Omega}{\rm div}\,u(u\cdot\nabla u\cdot n)dS-\frac{2\mu+\lambda}{2}\frac{d}{dt}\int({\rm div}\,u)^2dx
-(2\mu+\lambda)\int{\rm div}\,u{\rm div}\,(u\cdot\nabla u)dx\nonumber\\
&=-\frac{2\mu+\lambda}{2}\frac{d}{dt}\int(\div  u)^2dx
-(2\mu+\lambda)\int_{\partial\Omega}\div  u(u\cdot\nabla n\cdot u)dS
-(2\mu+\lambda)\int{\rm div}\,u\partial_j(u^i\partial_iu^j)dx\nonumber\\
&=-\frac{2\mu+\lambda}{2}\frac{d}{dt}\int({\rm div}\,u)^2dx
-(2\mu+\lambda)\int_{\partial\Omega}{\rm div}\,u(u\cdot\nabla n\cdot u)dS\nonumber\\
&\quad-(2\mu+\lambda)\int{\rm div}\,u\nabla u:\nabla udx-(2\mu+\lambda)\int{\rm div}\,uu^j\partial_j\partial_{i}u^idx\nonumber\\
&=-\frac{2\mu+\lambda}{2}\frac{d}{dt}\int({\rm div}\,u)^2dx
-(2\mu+\lambda)\int_{\partial\Omega}{\rm div}\,u(u\cdot\nabla n\cdot u)dS\nonumber\\
&\quad-(2\mu+\lambda)\int{\rm div}\,u\nabla u:\nabla udx+\frac{2\mu+\lambda}{2}\int({\rm div}\,u)^3dx
\nonumber\\
&\le -\frac{2\mu+\lambda}{2}\frac{d}{dt}\int({\rm div}\,u)^2dx
+\frac14\|\sqrt{\rho}\dot{u}\|_{L^2}^2+\frac18\|{\rm curl}^2H\|_{L^2}^2\nonumber\\
&\quad+C\big(\|\nabla u\|_{L^2}^4+\|\nabla H\|_{L^2}^4+\|\nabla u\|_{L^2}^2+\|\nabla H\|_{L^2}^2\big),
\end{align}
where we have used
\begin{align*}
\int{\rm div}\,uu^j\partial_j\partial_iu^idx&=-\int\partial_j(\partial_ku^ku^j)\partial_iu^idx=-\int\partial_j\partial_ku^ku^j\partial_iu^idx
-\int\divv u\partial_ju^j\partial_iu^idx,
\end{align*}
and
\begin{align}\label{3.24}
&\Big|-(2\mu+\lambda)\int_{\partial\Omega}{\rm div}\,u(u\cdot\nabla n\cdot u)dS\Big|\nonumber\\
&=\Big|-\int_{\partial\Omega}\Big(F+(P-\bar P)+\frac12|H|^2\Big)(u\cdot\nabla n\cdot u)dS\Big|\nonumber\\
&\le \Big|\int_{\partial\Omega}F(u\cdot\nabla n\cdot u)dS\Big|+\Big|\int_{\partial\Omega}(P-\bar P)(u\cdot\nabla n\cdot u)dS\Big|
+\frac12\Big|\int_{\partial\Omega}|H|^2(u\cdot\nabla n\cdot u)dS\Big|\nonumber\\
&\le C\int_{\partial\Omega}|F||u|^2dS+C\int_{\partial\Omega}|u|^2dS
+C\int_{\partial\Omega}|H|^2|u|^2dS
\nonumber\\
&\le C\big(\|\nabla F\|_{L^2}\|u\|_{L^4}^2+\|F\|_{L^6}\|u\|_{L^3}\|\nabla u\|_{L^2}
+\|F\|_{L^2}\|u\|_{L^4}^2\big)+C\|\nabla u\|_{L^2}^2\nonumber\\
&\quad+C\big(\|\nabla H\|_{L^2}\|H\|_{L^6}\|u\|_{L^6}^2+\|H\|_{L^6}^2\|u\|_{L^6}\|\nabla u\|_{L^2}+\|H\|_{L^4}^2\|u\|_{L^4}^2\big)\nonumber\\
&\le C\|F\|_{H^1}\|u\|_{H^1}^2+C\|\nabla u\|_{L^2}^2+C\|\nabla u\|_{L^2}^4+C\|\nabla H\|_{L^2}^4\nonumber\\
&\le \frac34\|\sqrt{\rho}\dot{u}\|_{L^2}^2+\frac18\|{\rm curl}^2H\|_{L^2}^2
+C\big(\|\nabla u\|_{L^2}^4+\|\nabla H\|_{L^2}^4+\|\nabla u\|_{L^2}^2+\|\nabla H\|_{L^2}^2\big),
\end{align}
due to the trace theorem, \eqref{z2.5}, \eqref{f6}, \eqref{f7}, \eqref{f8}, \eqref{f13}, and
\begin{align*}
\|H\cdot\nabla H\|_{L^2}
&\le C\|H\|_{L^6}\|\nabla H\|_{L^3} \\
&\le C\|\nabla H\|_{L^2}^\frac12\|\nabla H\|_{L^6}^\frac12 \\
&\le C\|\nabla H\|_{L^2}^\frac12\big(\|{\rm curl}^2H\|_{L^2}+\|\nabla H\|_{L^2}\big)^\frac12\\
&\le C\|\nabla H\|_{L^2}^\frac12\|{\rm curl}^2H\|_{L^2}^\frac12+\|\nabla H\|_{L^2}.
\end{align*}

Noting that
\begin{align*}
\int{\rm curl}\,u\cdot(u^i\partial_i{\rm curl}\,u)dx=-\int{\rm curl}\,u\cdot(u^i\partial_i{\rm curl}\,u)dx-\int|{\rm curl}\,u|^2{\rm div}\,udx,
\end{align*}
we have
\begin{align*}
\int{\rm curl}\,u\cdot(u^i\partial_i{\rm curl}\,u)dx
=-\frac12\int|{\rm curl}\,u|^2{\rm div}\,udx.
\end{align*}
This implies that
\begin{align*}
\int{\rm curl}\,u\cdot{\rm curl}\,(u\cdot\nabla u)dx&=\int{\rm curl}\,u\cdot{\rm curl}\,(u^i\partial_i u)dx
=\int{\rm curl}\,u\cdot\big(u^i{\rm curl}\,\partial_iu+\nabla u^i\times\partial_iu\big)dx\nonumber\\
&=-\int\partial_i({\rm curl}\,u u^i){\rm curl}\,udx+\int(\nabla u^i\times\partial_iu)\cdot{\rm curl}\,udx\nonumber\\
&=-\frac12\int|{\rm curl}\,u|^2{\rm div}\,udx+\int(\nabla u^i\times\partial_iu)\cdot{\rm curl}\,udx,
\end{align*}
which combined with \eqref{a4} and integration by parts leads to
\begin{align}\label{3.26}
I_3&=-\mu\int{\rm curl}\,u\cdot{\rm curl}\,\dot{u}dx\nonumber\\
&=-\frac{\mu}{2}\frac{d}{dt}\int|{\rm curl}\,u|^2dx
-\mu\int{\rm curl}\,u\cdot{\rm curl}\,(u\cdot\nabla u)dx\nonumber\\
&=-\frac{\mu}{2}\frac{d}{dt}\int|{\rm curl}\,u|^2dx
-\mu\int(\nabla u^i\times\partial_iu)\cdot{\rm curl}\,udx
+\frac{\mu}{2}\int|{\rm curl}\,u|^2{\rm div}\,udx\nonumber\\
&\le -\frac{\mu}{2}\frac{d}{dt}\int|{\rm curl}\,u|^2dx
+C\|\nabla u\|_{L^3}^3.
\end{align}

By \eqref{a4} and integration by parts, it indicates that
\begin{align*}
I_{4}&=\int\Big(H\cdot\nabla H-\frac12\nabla|H|^2\Big)\cdot u_tdx+\int\Big(H\cdot\nabla H-\frac12\nabla|H|^2\Big)\cdot u\cdot\nabla udx\nonumber\\
&=\frac{d}{dt}\int\Big(H\cdot\nabla H-\frac12\nabla|H|^2\Big)\cdot udx
+\int\Big((H\otimes H)_t:\nabla u-\frac12(|H|^2)_t{\rm div}\,u\Big)dx\nonumber\\
&\quad+\int\Big(H\cdot\nabla H-\frac12\nabla|H|^2\Big)\cdot u\cdot\nabla udx\nonumber\\
&\le \frac{d}{dt}\int\Big(H\cdot\nabla H-\frac12\nabla|H|^2\Big)\cdot udx
+C\|\nabla u\|_{L^3}\|H_t\|_{L^2}\|H\|_{L^6}\nonumber\\
&\quad+C\|H\|_{L^6}\big(\|{\rm curl}^2H\|_{L^2}+\|\nabla H\|_{L^2}\big)\|\nabla u\|_{L^2}\|u\|_{L^6}\nonumber\\
&\le \frac{d}{dt}\int\Big(H\cdot\nabla H-\frac12\nabla|H|^2\Big)\cdot udx
+C\delta(\|H_t\|_{L^2}^2+\|{\rm curl}^2H\|_{L^2}^2)\nonumber\\
&\quad+C\big(\|\nabla H\|_{L^2}^2\|\nabla u\|_{L^2}^2+\|\nabla u\|_{L^3}^3\big)
+C\big(\|\nabla H\|_{L^2}^6+\|\nabla H\|_{L^2}^2\|\nabla u\|_{L^2}^4\big).
\end{align*}

Multiplying $\eqref{a1}_3$ by $4|H|^2H$ and integration by parts, we have
\begin{align*}
\frac{d}{dt}\|H\|_{L^4}^4&=-4\int\curl H\cdot\curl(|H|^2H)dx+4\int|H|^2H\cdot\nabla u\cdot Hdx
-2\int|H|^4\divv udx\nonumber\\
&=-4\int|\curl H|^2|H|^2dx-4\int\curl H\cdot\nabla|H|^2\times Hdx-2\int|H|^4\divv udx\nonumber\\
&\quad+4\int|H|^2H\cdot\nabla u\cdot Hdx\nonumber\\
&\le -2\int|\curl H|^2|H|^2dx+C\int|H|^2|\nabla H|^2dx+C\int|H|^4|\nabla u|dx\nonumber\\
&\le -2\int|\curl H|^2|H|^2dx+C\|H\|_{L^6}^2\|\nabla H\|_{L^3}^2+C\|\nabla u\|_{L^2}\||H|^2\|_{L^6}\||H|^2\|_{L^3}\nonumber\\
&\le -2\int|\curl H|^2|H|^2dx+C\|\nabla H\|_{L^2}^3\|\curl^2 H\|_{L^2}+C\|\nabla H\|_{L^2}^4
+C\|\nabla u\|_{L^2}^2\|\nabla H\|_{L^2}^4\nonumber\\
&\le -2\int|\curl H|^2|H|^2dx+\frac18\|\curl^2H\|_{L^2}^2+C\|\nabla H\|_{L^2}^6
+C\|\nabla u\|_{L^2}^6+C\|\nabla H\|_{L^2}^4,
\end{align*}
which implies that
\begin{align}\label{w28}
&\frac{d}{dt}\|H\|_{L^4}^4+2\int|\curl H|^2|H|^2dx\le \frac18\|\curl^2H\|_{L^2}^2+C\|\nabla H\|_{L^2}^6
+C\|\nabla u\|_{L^2}^6+C\|\nabla H\|_{L^2}^4.
\end{align}

In view of $\eqref{a1}_3$, one deduces from Gagliardo--Nirenberg inequality, Young's inequality, and \eqref{a4} that
\begin{align*}
&\frac{d}{dt}\|{\rm curl}\,H\|_{L^2}^2+\|{\rm curl}^2H\|_{L^2}^2+\|H_t\|_{L^2}^2\nonumber\\
&=\int|H_t-\curl\curl H|^2dx=\int|H\cdot\nabla u-u\cdot\nabla H-H\div  u|^2dx\nonumber\\
&\le C\|\nabla u\|_{L^2}^2\|H\|_{L^\infty}^2+C\|u\|_{L^6}^2\|\nabla H\|_{L^3}^2\nonumber\\
&\le C\|\nabla u\|_{L^2}^2\|\nabla H\|_{L^2}\|{\rm curl}^2H\|_{L^2}+C\|\nabla u\|_{L^2}^2\|\nabla H\|_{L^2}^2\nonumber\\
&\le \frac18\|{\rm curl}^2H\|_{L^2}^2+C\big(\|\nabla u\|_{L^2}^2+\|\nabla u\|_{L^2}^4\big)\|\nabla H\|_{L^2}^2.
\end{align*}
This together with \eqref{w19} and \eqref{w28} yields that
\begin{align}\label{w30}
&\frac{d}{dt}\Psi(t)+\frac34\|\sqrt{\rho}\dot{u}\|_{L^2}^2
+\frac58\|{\rm curl}^2H\|_{L^2}^2+\|H_t\|_{L^2}^2+2D_2\int|\curl H|^2|H|^2dx\nonumber\\
&\le C\big(\|\nabla u\|_{L^3}^3+\|\nabla u\|_{L^2}^4+\|\nabla H\|_{L^2}^4+\|\nabla u\|_{L^2}^2+\|\nabla H\|_{L^2}^2
+\|\nabla H\|_{L^2}^6+\|\nabla u\|_{L^2}^6\big),
\end{align}
where
\begin{align*}
\Psi(t)&\triangleq\int\Big[\frac12\big((2\mu+\lambda)(\divv u)^2+\mu|\curl u|^2+2|\nabla H|^2\big)+D_2|H|^4\nonumber\\
&\quad-(P-\bar P)\divv u+H\otimes H:\nabla u-\frac12|H|^2\divv u\Big]dx.
\end{align*}
By \eqref{f12}, one has that
\begin{align}\label{w31}
\|\nabla u\|_{L^3}^3&\le C\big(\|\sqrt{\rho}\dot{u}\|_{L^2}+\|H\cdot\nabla H\|_{L^2}\big)^\frac{3}{2}\big(\|\nabla u\|_{L^2}+\|P-\bar P\|_{L^2}+\|H\|_{L^4}^2\big)^\frac32\nonumber\\
&\quad+C\big(\|\nabla u\|_{L^2}^3+\|P-\bar P\|_{L^3}^3+\|H\|_{L^6}^6\big)\nonumber\\
&\le  C\big(\|\sqrt{\rho}\dot{u}\|_{L^2}+\|H\cdot\nabla H\|_{L^2}\big)^\frac{3}{2}\big(\|\nabla u\|_{L^2}+\|\rho-\bar\rho_0\|_{L^2}+\|H\|_{L^4}^2\big)^\frac32\nonumber\\
&\quad+C\big(\|\nabla u\|_{L^2}^3+\|H\|_{L^6}^6
+\|\rho-\bar\rho_0\|_{L^2}^2\big)\nonumber\\
&\le \f14\|\sqrt{\rho}\dot{u}\|_{L^2}^2+C\big(\|H\cdot\nabla H\|_{L^2}^2
+\|\nabla u\|_{L^2}^6+\|H\|_{L^4}^{12}\big)\nonumber\\
&\quad+C\big(\|\nabla u\|_{L^2}^2+\|\nabla u\|_{L^2}^4+\|\nabla H\|_{L^2}^6\big)+C\|\rho-\bar\rho_0\|_{L^2}^2\nonumber\\
&\le \frac14\|\sqrt{\rho}\dot{u}\|_{L^2}^2+\frac18\|{\rm curl}^2H\|_{L^2}^2
+C\|\nabla H\|_{L^2}^4+C\|\nabla u\|_{L^2}^6\nonumber\\
&\quad+C\big(\|\nabla u\|_{L^2}^2+\|\nabla u\|_{L^2}^4+\|\nabla H\|_{L^2}^6\big)+C\|\rho-\bar\rho_0\|_{L^2}^2,
\end{align}
due to
\begin{align*}
\|P-\bar P\|_{L^3}\le C\|P-\bar P\|_{L^\infty}^\frac13\bigg(\int|P-\bar P|^2dx\bigg)^\frac13\le C(\hat{\rho})\|\rho-\bar\rho_0\|_{L^2}^\frac23.
\end{align*}
Plugging \eqref{w31} into \eqref{w30}, we arrive at
\begin{align*}
&\frac{d}{dt}\Psi(t)+\frac12\int\big(\rho|\dot u|^2+|\curl H|^2+|\curl^2H|^2+2|H_t|^2+4D_2|\curl H|^2|H|^2\big)dx\nonumber\\
&\le C\big(\|\nabla u\|_{L^2}^2+\|\nabla H\|_{L^2}^2
+\|\nabla H\|_{L^2}^6+\|\nabla u\|_{L^2}^6+\|\rho-\bar\rho_0\|_{L^2}^2\big)
\end{align*}
due to $\|(\nabla u,\nabla H)\|_{L^2}^4\leq C\big(\|(\nabla u,\nabla H)\|_{L^2}^2
+\|(\nabla u,\nabla H)\|_{L^2}^6\big)$.
By virtue of \eqref{a8}, \eqref{a9}, \eqref{w1}, and \eqref{w8}, one has that
\begin{align}\label{w33}
(\sqrt{\rho}\dot u, \ \nabla^2H, \ |\curl H||H|)\in L_{\loc}^2(0, \infty; L^2), \ \Psi(t)\in C[0, \infty),
\end{align}
which along with \eqref{w1} and \eqref{w8} implies that
\begin{align}\label{zzz}
&\int_0^\infty\int\Big[\frac12\big((2\mu+\lambda)(\divv u)^2+\mu|\curl u|^2+2|\nabla H|^2\big)+D_2|H|^4\nonumber\\
&\quad-(P-\bar P)\divv u+H\otimes H:\nabla u-\frac12|H|^2\divv u+D_2|\rho-\bar\rho_0|^2\Big]dxdt<\infty.
\end{align}

Next, for $t\ge 0$, we choose positive constants $D_2$ and $D_3$ and define the temporal energy functional
\begin{align}\label{w35}
\cm_2(t)&=D_3\cm_1(t)+\int\Big[\frac12\big((2\mu+\lambda)(\divv u)^2+\mu|\curl u|^2+2|\nabla H|^2\big)+D_2|H|^4\nonumber\\
&\quad-(P-\bar P)\divv u+H\otimes H:\nabla u-\frac12|H|^2\divv u+D_2|\rho-\bar\rho_0|^2\Big](t)dx.
\end{align}
Notice that $\cm_2(t)$ is equivalent to $\|(\rho-\bar\rho_0, \sqrt\rho u, \nabla u, \nabla H, |H|^2)(t)\|_{L^2}^2$ provided we choose $D_2$ and $D_3$ large enough. Fix a positive constant $\delta_1$ that may be small. In light of \eqref{w1}
and \eqref{zzz}, we conclude that there exists a positive constant $T_1>0$ such that
\begin{align*}
\cm_2(T_1)<\delta_1.
\end{align*}

Now we claim that, for any $t\ge T_1$,
\begin{align}\label{b35}
\Psi(t)+D_2\|(\rho-\bar\rho_0)(\cdot,t)\|_{L^2}^2<2\delta_1.
\end{align}
This implies that, for any $t\ge T_1$,
\begin{align}\label{b36}
\int\big[\mu|\nabla u|^2+(\mu+\lambda)(\divv u)^2+|\nabla H|^2+D_2|H|^4\big](t)dx<4\delta_1.
\end{align}
Let $\delta_1$ be small enough, by \eqref{w15} and \eqref{b35}, one has that, for any $t\ge T_1$,
\begin{align}\label{b37}
\frac{d}{dt}\cm_2(t)+\frac{\cm_2(t)}{D_3}+\frac{\|\sqrt\rho\dot{u}\|_{L^2}^2
+\|\curl^2H\|_{L^2}^2+\||H||\curl H|\|_{L^2}^2}{D_3}\le 0.
\end{align}
Integrating \eqref{b37} with respect to $t$ over $[T_1, \infty)$ gives \eqref{w17}.

It remains to prove \eqref{b35}. If \eqref{b35} does not hold, owing to \eqref{w33}, there exists a time $T_2>T_1$ such that
\begin{align}\label{b38}
\Psi(t)+D_2\|(\rho-\bar\rho_0)(\cdot,t)\|_{L^2}^2=2\delta_1.
\end{align}
Taking a minimal value of $T_2$ satisfying \eqref{b38}, then \eqref{b35} holds for any $T_1\le t<T_2$. Integrating \eqref{b37} from
$T_1$ to $T_2$ implies that
\begin{align*}
\cm_2(T_2)\le \cm_2(T_1)<\delta_1,
\end{align*}
which contradicts to \eqref{b38}.
\end{proof}

\begin{lemma}
Under the assumptions of Theorem \ref{thm1}, there exist two positive constant $C_5$ and $\eta_5$, which are dependent on
$M$, and $K$, but independent of $t$, such that for any $t\ge 0$,
\begin{align}\label{w41}
\sup_{t\ge 0}\big(\|\sqrt\rho\dot u\|_{L^2}^2+\|H_t\|_{L^2}^2\big)
+\int_0^\infty\big(\|\nabla\dot u\|_{L^2}^2+\|\nabla H_t\|_{L^2}^2\big)dt\le C,
\end{align}
and
\begin{align}\label{w42}
\sup_{t\ge 0}\big(\|\sqrt\rho\dot u\|_{L^2}^2+\|H_t\|_{L^2}^2\big)\le C_5e^{-\eta_5t}.
\end{align}
\end{lemma}
\begin{proof}[Proof]
Taking \eqref{f5}, we rewrite $\eqref{a1}_2$ as
\begin{align}\label{3.41}
\rho\dot{u}=\nabla F-\mu\curl{\rm curl}\,u+\div (H\otimes H).
\end{align}
Applying $\dot{u}^j[\partial/\partial t+\div (u\cdot)]$ to the $j$th-component of $\eqref{3.41}$,
 and then integrating the resulting equality over $\Omega$, we get that
\begin{align}\label{w44}
\frac12\frac{d}{dt}\int\rho|\dot{u}|^2dx
&=\int\big(\dot{u}\cdot\nabla F_t+\dot{u}^j{\rm div}\,(u\partial_jF)\big)dx
-\mu\int\big(\dot{u}\cdot\curl{\rm curl}\,u_t+\dot{u}^j{\rm div}\,((\curl{\rm curl}\,u)^ju)\big)dx\nonumber\\
&\quad+\int\big(\dot{u}\cdot{\rm div}\,(H\otimes H)_t+\dot{u}^j{\rm div}\,({\rm div}\,(H\otimes H^j)u)\big)dx
\triangleq\sum_{i=1}^3J_i.
\end{align}
We denote by $h\triangleq u\cdot(\nabla n+(\nabla n)^{tr})$ and $u^\bot\triangleq-u\times n$, then it deduces from Lemma \ref{l23} that
\begin{align}\label{w43}
&-\int_{\partial\Omega}F_t(u\cdot\nabla n\cdot u)dS\nonumber\\
&=-\frac{d}{dt}\int_{\partial\Omega}(u\cdot\nabla n\cdot u)FdS
+\int_{\partial\Omega}Fh\cdot\dot{u}dS-\int_{\partial\Omega}Fh\cdot(u\cdot\nabla u)dS\nonumber\\
&=-\frac{d}{dt}\int_{\partial\Omega}(u\cdot\nabla n\cdot u)FdS
+\int_{\partial\Omega}Fh\cdot\dot{u}dS-\int_{\partial\Omega}Fh^i(\nabla u^i\times u^\bot)\cdot ndS\nonumber\\
&=-\frac{d}{dt}\int_{\partial\Omega}(u\cdot\nabla n\cdot u)FdS
+\int_{\partial\Omega}Fh\cdot\dot{u}dS
+\int Fh^i\nabla\times u^\bot\cdot\nabla u^idx
-\int\nabla u^i\times u^\bot\cdot\nabla(Fh^i)dx\nonumber\\
&\le-\frac{d}{dt}\int_{\partial\Omega}(u\cdot\nabla n\cdot u)FdS
+C\|\nabla F\|_{L^2}\|u\|_{L^3}\|\dot{u}\|_{L^6}\nonumber\\
&\quad+C\big(\|F\|_{L^3}\|u\|_{L^6}\|\nabla\dot{u}\|_{L^2}
+\|F\|_{L^3}\|u\|_{L^6}\|\dot{u}\|_{L^2}
+\|F\|_{L^3}\|\nabla u\|_{L^2}\|\dot{u}\|_{L^6}\big)\nonumber\\
&\quad+C\big(\|\nabla u\|_{L^2}\|u\|_{L^6}^2\|\nabla F\|_{L^6}+\|\nabla u\|_{L^4}^2\|u\|_{L^6}\|F\|_{L^3}\big)\nonumber\\
&\le -\frac{d}{dt}\int_{\partial\Omega}(u\cdot\nabla n\cdot u)FdS\nonumber\\
&\quad+C\big(\|\rho\dot{u}\|_{L^2}+\|P-\bar P\|_{L^2}+\|\nabla u\|_{L^2}+\|H\|_{L^3}\|\nabla H\|_{L^6}+\|\nabla u\|_{L^2}^2
+\|H\|_{L^4}^2\big)\nonumber\\
&\quad\cdot\|\nabla u\|_{L^2}\big(\|\nabla\dot{u}\|_{L^2}+\|\nabla u\|_{L^2}^2+\|\nabla u\|_{L^4}^2\big)
+C\|\nabla u\|_{L^2}^3\|\nabla F\|_{L^6}\nonumber\\
&\le -\frac{d}{dt}\int_{\partial\Omega}(u\cdot\nabla n\cdot u)FdS
+C\|\nabla u\|_{L^2}^3\|\nabla F\|_{L^6}+\delta\|\nabla\dot{u}\|_{L^2}^2\nonumber\\
&\quad+C\big(\|\sqrt{\rho}\dot{u}\|_{L^2}^2\|\nabla u\|_{L^2}^2+\|\nabla u\|_{L^2}^6+\|\nabla u\|_{L^4}^4
+\|\nabla u\|_{L^2}^2
\big)\nonumber\\
&\quad+C\big(\|\nabla u\|_{L^2}^4+\|\nabla H\|_{L^2}^4+\|{\rm curl}^2H\|_{L^2}^2\|\nabla u\|_{L^2}^2\|\na H\|_{L^2}^2
+\|\nabla H\|_{L^2}^6\big),
\end{align}
due to
\begin{align*}
{\rm div}(\nabla u^i\times u^\bot)&=u^\bot\cdot\curl\nabla u^i-\nabla u^i\cdot\curl u^\bot=-\nabla u^i\cdot\curl u^\bot,\\
\|\dot{u}\|_{L^6}&\le C\big(\|\nabla\dot{u}\|_{L^2}+\|\nabla u\|_{L^2}^2\big),\\
\|\sqrt\rho\dot u\|_{L^2}&\le C\|\dot u\|_{L^2}\le C\|\nabla\dot u\|_{L^2},
\\
\|\nabla F\|_{L^6}&\le C(\|\rho\dot u\|_{L^6}+\|H\cdot\nabla H\|_{L^6})\nonumber\\
&\le C(\|\dot u\|_{L^6}+\|H\|_{L^\infty}\|\nabla H\|_{L^6})\nonumber\\
&\le C(\|\nabla\dot u\|_{L^2}+\|\nabla H\|_{L^2}^\frac12\|\curl^2H\|_{L^2}^\frac32+\|\nabla u\|_{L^2}^2+\|\na H\|_{L^2}^2),
\end{align*}
and
\begin{align}
\|\nabla F\|_{L^2}&\le C(\|\rho\dot u\|_{L^2}+\|H\cdot\nabla H\|_{L^2})\nonumber\\
&\le C\|\sqrt\rho\dot u\|_{L^2}+C\|H\|_{L^6}\|\nabla H\|_{L^3}\nonumber\\
&\le C\|\sqrt\rho\dot u\|_{L^2}+C\|\nabla H\|_{L^2}^\frac12\|\nabla H\|_{L^6}^\frac12\nonumber\\
&\le C\|\sqrt\rho\dot u\|_{L^2}+C\|\nabla H\|_{L^2}^\f12\|\curl^2 H\|_{L^2}^\f12+C\|\na H\|_{L^2}.
\end{align}
Thus, it follows from integration by parts, H\"older's inequality, Gagliardo--Nirenberg inequality, \eqref{a15}, \eqref{f8}, and \eqref{z2.5} that
\begin{align}\label{wz1}
J_1&=\int\big(\dot{u}\cdot\nabla F_t+\dot{u}^j\pa_i(\pa_jFu^i)\big)dx\nonumber\\
&=\int_{\pa\Omega}F_t\dot u\cdot ndS-\int F_t\divv\dot udx-\int u^i\pa_i\dot u^j\partial_jFdx\nonumber\\
&=-\int_{\pa\Omega}F_t(u\cdot\nabla n\cdot u)dS-(2\mu+\lambda)\int(\divv\dot u)^2dx
+(2\mu+\lambda)\int\divv\dot u\nabla u:\nabla udx\nonumber\\
&\quad-\gamma\int P\divv\dot u\divv udx+\int\divv\dot uu\cdot\nabla Fdx-\int u^i\pa_i\dot u^j\partial_jFdx
+\int\divv\dot u H\cdot H_tdx\nonumber\\
&\quad+\int\divv\dot u u\cdot\nabla H\cdot Hdx\nonumber\\
&\le -\int_{\pa\Omega}F_t(u\cdot\nabla n\cdot u)dS-(2\mu+\lambda)\|\divv\dot u\|_{L^2}^2+\delta\|\nabla\dot u\|_{L^2}^2
+C\|\nabla u\|_{L^4}^4\nonumber\\
&\quad+C\|u\|_{L^6}^2\|\nabla F\|_{L^3}^2+C\|\nabla u\|_{L^2}^2+C\|H_t\|_{L^3}^2\|H\|_{L^6}^2
+C\|u\|_{L^6}^2\|\nabla H\|_{L^6}^2\|H\|_{L^6}^2\nonumber\\
&\le -\int_{\pa\Omega}F_t(u\cdot\nabla n\cdot u)dS
-(2\mu+\lambda)\|\divv\dot u\|_{L^2}^2+\delta\|\nabla\dot u\|_{L^2}^2
+C\|\nabla u\|_{L^4}^4\nonumber\\
&\quad+C\|\nabla u\|_{L^2}^2\|\nabla F\|_{L^2}\|\nabla F\|_{L^6}
+C\|\nabla u\|_{L^2}^2+C\|H_t\|_{L^2}\|\nabla H_t\|_{L^2}\|\nabla H\|_{L^2}^2\nonumber\\
&\quad+C\|\nabla u\|_{L^2}^2\|\curl^2H\|_{L^2}^2\|\nabla H\|_{L^2}^2+C\|\nabla u\|_{L^2}^2\|\nabla H\|_{L^2}^4\nonumber\\
&\le -\int_{\pa\Omega}F_t(u\cdot\nabla n\cdot u)dS-(2\mu+\lambda)\|\divv\dot u\|_{L^2}^2+\delta\|\nabla\dot u\|_{L^2}^2
+\delta\|\na H_t\|_{L^2}^2\nonumber\\
&\quad+C\|H_t\|_{L^2}^2\|\na H\|_{L^2}^4+C\|\nabla u\|_{L^4}^4
+C\|\nabla u\|_{L^2}^2\|\nabla F\|_{L^2}\|\nabla F\|_{L^6}
+C\|\nabla u\|_{L^2}^2\nonumber\\
&\quad+C\|\curl^2H\|_{L^2}^2(\|\na u\|_{L^2}^4+\|\nabla H\|_{L^2}^4)\nonumber\\
&\le -\frac{d}{dt}\int_{\pa\Om}(u\cdot\na n\cdot u)FdS-(2\mu+\lambda)\|\divv\dot u\|_{L^2}^2+\delta\|\nabla\dot u\|_{L^2}^2
+\delta\|\na H_t\|_{L^2}^2\nonumber\\
&\quad
+C\big(\|\sqrt{\rho}\dot{u}\|_{L^2}^2\|\nabla u\|_{L^2}^2+\|\nabla u\|_{L^2}^6+\|\nabla u\|_{L^4}^4
+\|\nabla u\|_{L^2}^2
\big)\nonumber\\
&\quad+C\big(\|\nabla u\|_{L^2}^4+\|\nabla H\|_{L^2}^4
+\|\nabla H\|_{L^2}^6\big)+C\|\curl^2H\|_{L^2}^2(\|\na u\|_{L^2}^4+\|\nabla H\|_{L^2}^4)\nonumber\\
&\quad+C(\|\na u\|_{L^2}^3+\|\na u\|_{L^2}^2\|\na F\|_{L^2})\|\na F\|_{L^6}\nonumber\\
&\le -\frac{d}{dt}\int_{\pa\Om}(u\cdot\na n\cdot u)FdS-(2\mu+\lambda)\|\divv\dot u\|_{L^2}^2+\delta\|\nabla\dot u\|_{L^2}^2
+\delta\|\na H_t\|_{L^2}^2\nonumber\\
&\quad
+C\big(\|\sqrt{\rho}\dot{u}\|_{L^2}^2\|\nabla u\|_{L^2}^2+\|\nabla u\|_{L^2}^6+\|\nabla H\|_{L^2}^6+\|\nabla u\|_{L^4}^4
+\|\nabla u\|_{L^2}^2+\|\na H\|_{L^2}^2\big)\nonumber\\
&\quad+C\|\curl^2H\|_{L^2}^2(\|\na u\|_{L^2}^4+\|\nabla H\|_{L^2}^4+\|\na H\|_{L^2}^2)+C\|\na u\|_{L^2}^4\|\sqrt\r\dot u\|_{L^2}^2,
\end{align}
where we have used
\begin{align*}
\|\na u\|_{L^2}^3\|\na F\|_{L^6}
&\le \|\na u\|_{L^2}^3(\|\na \dot u\|_{L^2}+\|\na H\|_{L^2}^\f12\|\curl^2 H\|_{L^2}^\f32+\|\na u\|_{L^2}^2+\|\na H\|_{L^2}^2)\nonumber\\
&\le \delta\|\na\dot u\|_{L^2}^2+C\|\na u\|_{L^2}^6+C\|\na u\|_{L^2}^4\|\curl^2H\|_{L^2}^2+C\|\na u\|_{L^2}^4\notag \\
&\quad+C\|\na H\|_{L^2}^4
+C\|\na H\|_{L^2}^2,\\
\|\na u\|_{L^2}^2\|\na F\|_{L^2}\|\na F\|_{L^6}
&\le C\|\na u\|_{L^2}^2(\|\sqrt\r\dot u\|_{L^2}+\|\na H\|_{L^2}^\f12\|\curl^2H\|_{L^2}^\f12+\|\na H\|_{L^2})\nonumber\\
&\quad\times(\|\na\dot u\|_{L^2}+\|\na H\|_{L^2}^\f12\|\curl^2H\|_{L^2}^\f32+\|\na u\|_{L^2}^2+\|\na H\|_{L^2}^2)\nonumber\\
&\le \delta\|\na\dot u\|_{L^2}^2+C\|\na u\|_{L^2}^4\|\sqrt\r\dot u\|_{L^2}^2
+C\|\na u\|_{L^2}^4\|\curl^2H\|_{L^2}^2+C\|\na u\|_{L^2}^6\nonumber\\
&\quad+C\|\na H\|_{L^2}^6+C\|\na H\|_{L^2}^2\|\curl^2H\|_{L^2}^2
+C\|\na u\|_{L^2}^4+C\|\na H\|_{L^2}^4,
\end{align*}
and
\begin{align*}
F_{t}+u\cdot\nabla F&=(2\mu+\lambda)\divv u_t-P_t-H\cdot H_t
+(2\mu+\lambda)u\cdot\nabla\divv u-u\cdot\nabla P-u\cdot\nabla H\cdot H\nonumber\\
&=(2\mu+\lambda)\divv\dot{u}-(2\mu+\lambda)\divv(u\cdot\nabla u)+(2\mu+\lambda)u\cdot\nabla\divv u\nonumber\\
&\quad-(P_t+u\cdot\nabla P)-H\cdot H_t-u\cdot\nabla H\cdot H\nonumber\\
&=(2\mu+\lambda)\divv\dot{u}-(2\mu+\lambda)\partial_i(u^j\partial_ju^i)
+(2\mu+\lambda)u^j\partial_j\partial_iu^i\nonumber\\
&\quad+\gamma P\divv u-H\cdot H_t-u\cdot\nabla H\cdot H\nonumber\\
&=(2\mu+\lambda)\divv\dot{u}-(2\mu+\lambda)\partial_iu^j\partial_ju^i+\gamma P\divv u-H\cdot H_t-u\cdot\nabla H\cdot H.
\end{align*}

By direct calculations, we have
\begin{align}\label{wz2}
J_2&=-\mu\int \dot{u}\cdot(\curl{\rm curl}\,u_t)dx-\mu\int\dot{u}\cdot(\curl{\rm curl}\,u){\rm div}\,udx
-\mu\int u^i\dot{u}\cdot\curl(\partial_i{\rm curl}\,u)dx\nonumber\\
&=-\mu\int |{\rm curl}\,\dot{u}|^2dx+\mu\int {\rm curl}\,\dot{u}\cdot{\rm curl}\,(u\cdot\nabla u)dx+\mu\int ({\rm curl}\,u\times\dot{u})\cdot\nabla{\rm div}\,udx\nonumber\\
&\quad
-\mu\int {\rm div}\,u({\rm curl}\,u\cdot{\rm curl}\,\dot{u})dx-\mu\int u^i{\rm div}\,(\partial_i{\rm curl}\,u\times\dot{u})dx
-\mu\int u^i\partial_i{\rm curl}\,u\cdot{\rm curl}\,\dot{u}dx\nonumber\\
&=-\mu\int |{\rm curl}\,\dot{u}|^2dx+\mu\int {\rm curl}\,\dot{u}\partial_iu\times\nabla u^idx+\mu\int ({\rm curl}\,u\times\dot{u})\cdot\nabla{\rm div}\,udx\nonumber\\
&\quad
-\mu\int {\rm div}\,u({\rm curl}\,u\cdot{\rm curl}\,\dot{u})dx-\mu\int u^i{\rm div}\,(\partial_i{\rm curl}\,u\times\dot{u})dx\nonumber\\
&=-\mu\int |{\rm curl}\,\dot{u}|^2dx+\mu\int {\rm curl}\,\dot{u}\partial_iu\times\nabla u^idx+\mu\int ({\rm curl}\,u\times\dot{u})\cdot\nabla{\rm div}\,udx
\nonumber\\
&\quad-\mu\int {\rm div}\,u({\rm curl}\,u\cdot{\rm curl}\,\dot{u})dx-\mu\int  u\cdot\nabla{\rm div}\,({\rm curl}\,u\times\dot{u})dx
+\mu\int  u^i{\rm div}\,({\rm curl}\,u\times\partial_i\dot{u})dx\nonumber\\
&=-\mu\int |{\rm curl}\,\dot{u}|^2dx+\mu\int {\rm curl}\,\dot{u}\nabla_iu\times\nabla u^idx\nonumber\\
&\quad-\mu\int {\rm div}\,u({\rm curl}\,u\cdot{\rm curl}\,\dot{u})dx
-\mu\int \nabla u^i\cdot({\rm curl}\,u\times\partial_i\dot{u})dx\nonumber\\
&\le \delta \|\nabla\dot{u}\|_{L^2}^2+C \|\nabla u\|_{L^4}^4-\mu \|{\rm curl}\,\dot{u}\|_{L^2}^2,
\end{align}
due to
\begin{align*}
\curl(\dot{u}{\rm div}\,u)&={\rm div}\,u{\rm curl}\,\dot{u}+\nabla{\rm div}\,u\times\dot{u},\\
{\rm div}\,(\partial_i{\rm curl}\,u\times\dot{u})&=\dot{u}\cdot\curl(\partial_i{\rm curl}\,u)-\partial_i{\rm curl}\,u\cdot\curl\dot{u},\\
\int{\rm curl}\,u\cdot(\nabla{\rm div}\,u\times\dot{u})dx&=-\int({\rm curl}\,u\times\dot{u})\cdot\nabla{\rm div}\,udx,\\
\int{\rm curl}\,\dot{u}\cdot{\rm curl}\,(u\cdot\nabla u)dx&=\int{\rm curl}\,\dot{u}\cdot{\rm curl}\,(u^i\partial_iu)dx \\
& =\int{\rm curl}\,\dot{u}\big(u^i{\rm curl}\,\partial_iu+\partial_iu\cdot\nabla u^i\big)dx\nonumber\\
&=\int u^i\partial_i{\rm curl}\,u\cdot{\rm curl}\,\dot{u}dx+\int{\rm curl}\,\dot{u}\partial_iu\times\nabla u^idx,
\end{align*}
and
\begin{align*}
\int u\cdot\nabla{\rm div}\,({\rm curl}\,u\times\dot{u})dx
&=\int u^i\partial_i{\rm div}\,({\rm curl}\,u\times\dot{u})dx\nonumber\\
&=\int u^i\partial_i(\dot{u}\cdot\curl{\rm curl}\,u-{\rm curl}\,\dot{u}\cdot{\rm curl}\,u)dx\nonumber\\
&=\int u^i(\dot{u}\cdot\partial_i\curl{\rm curl}\,u-{\rm curl}\,\dot{u}\cdot\partial_i{\rm curl}\,u)dx\nonumber\\
&\quad+\int u^i(\partial_i\dot{u}\cdot\curl{\rm curl}\,u-\partial_i{\rm curl}\,\dot{u}\cdot{\rm curl}\,u)dx\nonumber\\
&=\int u^i{\rm div}\,(\partial_i{\rm curl}\,u\times\dot{u})dx+\int u^i{\rm div}\,({\rm curl}\,u\times\partial_i\dot{u})dx.
\end{align*}
By virtue of H\"older's, Young's, and Sobolev's inequalities, we have
\begin{align}\label{wz3}
J_3&=-\int\na\dot u:(H\otimes H)_tdx-\mu\int H^i\pa_iH^ju^k\pa_k\dot u^jdx\nonumber\\
&\le C(\|\na\dot u\|_{L^2}\|H\|_{L^6}\|H_t\|_{L^3}+\|\na\dot u\|_{L^2}\|H\|_{L^6}\|\nabla H\|_{L^6}\|u\|_{L^6})\nonumber\\
&\le \delta\|\na\dot u\|_{L^2}^2+\delta\|\na H_t\|_{L^2}^2+C\|\na H\|_{L^2}^4\|H_t\|_{L^2}^2
+C\|\na H\|_{L^2}^4+C\|\nabla u\|_{L^2}^4\nonumber\\
&\quad+C\|\curl^2H\|_{L^2}^2(\|\na u\|_{L^2}^4+\|\nabla H\|_{L^2}^4).
\end{align}
Plugging \eqref{wz1}--\eqref{wz3} into \eqref{w44}, we obtain after choosing $\delta$ suitably small that
\begin{align}\label{w51}
&\f{d}{dt}\|\sqrt\r\dot u\|_{L^2}^2+(2\mu+\lambda)\|\divv\dot u\|_{L^2}^2+\mu\|\curl\dot u\|_{L^2}^2\nonumber\\
&\le  -\frac{d}{dt}\int_{\pa\Om}(u\cdot\na n\cdot u)FdS+C\|\na u\|_{L^2}^4\|\sqrt\r\dot u\|_{L^2}^2+C\|\na H\|_{L^2}^4\|H_t\|_{L^2}^2\nonumber\\
&\quad+C\big(\|\sqrt{\rho}\dot{u}\|_{L^2}^2\|\nabla u\|_{L^2}^2+\|\nabla u\|_{L^2}^6+\|\nabla H\|_{L^2}^6+\|\nabla u\|_{L^4}^4
+\|\nabla u\|_{L^2}^2+\|\na H\|_{L^2}^2\big)\nonumber\\
&\quad+C\|\curl^2H\|_{L^2}^2(\|\na u\|_{L^2}^4+\|\nabla H\|_{L^2}^4+\|\na H\|_{L^2}^2).
\end{align}

Note that
\begin{align}\label{z3.36}
\begin{cases}
H_{tt}-\curl{\rm curl}\,H_t=\big(H\cdot\nabla u-u\cdot\nabla H-H\cdot\nabla u\big)_t \quad&{\rm in}~\Omega,\\
H_t\cdot n=0, \quad {\rm curl}\,H_t\times n=0 \quad&{\rm on}~\partial\Omega.
\end{cases}
\end{align}
Multiplying \eqref{z3.36}$_1$ by $H_t$ and integration by parts, we find that
\begin{align}\label{w53}
\frac{d}{dt}\|H_t\|_{L^2}^2+\|{\rm curl}\,H_t\|_{L^2}^2
&=\int (H_t\cdot\nabla u-u\cdot\nabla H_t-H_t{\rm div}\,u)\cdot H_tdx\nonumber\\
&\quad+\int (H\cdot\nabla\dot{u}-\dot{u}\cdot\nabla H-H{\rm div}\,\dot{u})\cdot H_tdx\nonumber\\
&\quad-\int (H\cdot\nabla(u\cdot\nabla u)-u\cdot\nabla u\cdot\nabla H-H{\rm div}\,(u\cdot\nabla u))\cdot H_tdx
\triangleq\sum_{i=1}^3R_i.
\end{align}
One gets from H\"older's inequality, Young's inequality, Gagliardo--Nirenberg inequality, \eqref{a15}, and \eqref{z2.5} that
\begin{align*}
R_1&\le C(\|H_t\|_{L^3}\|H_t\|_{L^6}\|\na u\|_{L^2}+\|u\|_{L^6}\|H_t\|_{L^3}\|\na H_t\|_{L^2})\nonumber\\
&\le C\|H_t\|_{L^2}^\f12\|\na H_t\|_{L^2}^\f32\|\nabla u\|_{L^2}\nonumber\\
&\le \delta\|\na H_t\|_{L^2}^2+C\|\na u\|_{L^2}^4\|H_t\|_{L^2}^2,\\
R_2&\le C \|H\|_{L^6}\|H_t\|_{L^3}\|\nabla\dot{u}\|_{L^2}+C \|\dot{u}\|_{L^6}\|\nabla H\|_{L^2}\|H_t\|_{L^3}\nonumber\\
&\le C \big(\|\nabla\dot{u}\|_{L^2}
+\|\nabla u\|_{L^2}^2\big)\|\nabla H\|_{L^2}\|H_t\|_{L^2}^\frac12\|\nabla H_t\|_{L^2}^\frac12\nonumber\\
&\le \delta \big(\|\nabla\dot{u}\|_{L^2}^2+\|\nabla H_t\|_{L^2}^2\big)+C \|\nabla u\|_{L^2}^4
+C \|\nabla H\|_{L^2}^4\|H_t\|_{L^2}^2,\\
R_3&=-\int_{\partial\Omega} H\cdot H_t(u\cdot\nabla u\cdot n)dS+\int H\cdot\nabla H_t\cdot(u\cdot\nabla u)dx
-\int (u\cdot\nabla u)\cdot\nabla H\cdot H_tdx\nonumber\\
&\quad+\int u\cdot\nabla u\cdot\nabla H_t\cdot Hdx+\int u\cdot\nabla u\cdot\nabla H\cdot H_tdx\nonumber\\
&\le \int_{\partial\Omega} H\cdot H_t(u\cdot\nabla n\cdot u)dS
+2 \|H\|_{L^{12}}\|\nabla H_t\|_{L^2}\|u\|_{L^6}\|\nabla u\|_{L^4}\nonumber\\
&\le \int_{\partial\Omega} H\cdot H_t(u\cdot\nabla n\cdot u)dS
+\delta \|\nabla H_t\|_{L^2}^2+C \|\nabla u\|_{L^2}^4\|{\rm curl}^2H\|_{L^2}^2\nonumber\\
&\quad+C \|\nabla u\|_{L^4}^4+C \big(\|\nabla u\|_{L^2}^4+\|\nabla H\|_{L^2}^4\big),
\end{align*}
where we have used
\begin{align*}
\|H\|_{L^{12}}&\le C\||H|^2\|_{H^1}^\f12\le C\|H\|_{L^4}+C\|\nabla |H|^2\|_{L^2}^{\frac12} \\
& \le C\|\nabla H\|_{L^2}+C\|H\|_{L^6}^{\frac12}\|\nabla H\|_{L^3}^{\frac12} \\
& \le C\|\nabla H\|_{L^2}+C\|H\|_{L^6}^{\frac12}\|\nabla H\|_{L^2}^\frac14\|\nabla H\|_{L^6}^\frac14 \\
& \le C\|\nabla H\|_{L^2}+C\|\nabla H\|_{L^2}^\frac14\|{\rm curl}^2H\|_{L^2}^\frac14.
\end{align*}
Thus, substituting the above estimates on $R_i\ (i=1, 2, 3)$ into \eqref{w53} shows that
\begin{align}\label{w54}
\frac{d}{dt}\|H_t\|_{L^2}^2+\|{\rm curl}\, H_t\|_{L^2}^2
&\le \int_{\partial\Omega}(H\cdot H_t)(u\cdot\nabla n\cdot u)dS+ C\tilde\delta \big(\|\nabla\dot{u}\|_{L^2}^2+\|\nabla H_t\|_{L^2}^2\big)
+C \|\nabla H\|_{L^2}^4\|H_t\|_{L^2}^2\nonumber\\
&\quad+C \|\nabla u\|_{L^4}^4+C \|\nabla u\|_{L^2}^4\|{\rm curl}^2H\|_{L^2}^2+C \big(\|\nabla u\|_{L^2}^4+\|\nabla H\|_{L^2}^4\big).
\end{align}
By Lemma \ref{l23}, H{\"o}lder's inequality, Sobolev's inequality, and \eqref{z2.5}, we find that
\begin{align*}
\left|\int_{\partial\Omega}(H\cdot H_t)(u\cdot\nabla n\cdot u)dS\right|
&\le C\int_{\partial\Omega}|u|^2|H||H_t|dS\le C\||u|^2|H||H_t|\|_{W^{1,1}} \\
&\le C\big(\|u\|_{L^4}^2\|H\|_{L^3}\|H_t\|_{L^6}
+\|u\|_{L^6}\|\nabla u\|_{L^2}\|H\|_{L^6}\|H_t\|_{L^6}\big) \\
&\quad+C\big(\|u\|_{L^6}^2\|\nabla H_t\|_{L^2}\|H\|_{L^6}+\|u\|_{L^6}^2\|\nabla H\|_{L^2}\|H_t\|_{L^6}\big) \\
&\le C\|u\|_{H^1}^2\|H\|_{H^1}\|H_t\|_{H^1}
\le C\|\nabla u\|_{L^2}^2\|\nabla H\|_{L^2}\|\nabla H_t\|_{L^2} \\
&\le \tilde\delta\|\nabla H_t\|_{L^2}^2+C\|\nabla u\|_{L^2}^4\|\nabla H\|_{L^2}^2,
\end{align*}
which along with \eqref{w54} implies that
\begin{align}\label{z6}
\frac{d}{dt}\|H_t\|_{L^2}^2+\|{\rm curl}\, H_t\|_{L^2}^2
&\le C\tilde\delta  \big(\|\nabla\dot{u}\|_{L^2}^2+\|\nabla H_t\|_{L^2}^2\big)
+C \|\nabla H\|_{L^2}^4\|H_t\|_{L^2}^2+C \|\nabla u\|_{L^2}^4\|{\rm curl}^2H\|_{L^2}^2\nonumber\\
&\quad+C \big(\|\nabla u\|_{L^2}^4+\|\nabla H\|_{L^2}^4\big)
+C \|\nabla u\|_{L^4}^4+C \|\nabla u\|_{L^2}^4\|\nabla H\|_{L^2}^2.
\end{align}
Hence, by choosing $\tilde\delta$ suitably small, one obtains from \eqref{w51} and \eqref{z6} that
\begin{align}\label{w55}
&\frac{d}{dt}\lb\|\sqrt\r\dot u\|_{L^2}^2+\|H_t\|_{L^2}^2+\int_{\pa\Om}(u\cdot\na n\cdot u)FdS
\rb+\|\na\dot u\|_{L^2}^2+\|{\rm curl}\, H_t\|_{L^2}^2\nonumber\\
&\le C\|\na u\|_{L^2}^4\|\sqrt\r\dot u\|_{L^2}^2+C\|\na H\|_{L^2}^4\|H_t\|_{L^2}^2
+C\big(\|\nabla u\|_{L^2}^4+\|\nabla H\|_{L^2}^4
\big)\nonumber\\
&\quad+C\big(\|\sqrt{\rho}\dot{u}\|_{L^2}^2\|\nabla u\|_{L^2}^2+\|\nabla u\|_{L^2}^6+\|\nabla H\|_{L^2}^6+\|\nabla u\|_{L^4}^4
+\|\nabla u\|_{L^2}^2+\|\na H\|_{L^2}^2\big)\nonumber\\
&\quad+C\|\curl^2H\|_{L^2}^2\big(\|\na u\|_{L^2}^4+\|\nabla H\|_{L^2}^4+\|\na H\|_{L^2}^2\big)\nonumber\\
&\le C\big(\|\sqrt\rho\dot u\|_{L^2}^2+\|H_t\|_{L^2}^2+\|\na u\|_{L^2}^2+\|\na H\|_{L^2}^2\big)\big(\|\sqrt\rho\dot u\|_{L^2}^2+\|H_t\|_{L^2}^2\big)\nonumber\\
&\quad+C\big(\|\nabla u\|_{L^2}^2+\|\na H\|_{L^2}^2+\|\rho-\bar\rho_0\|_{L^2}^2\big),
\end{align}
where we have used the following facts
\begin{align}
\|\na u\|_{L^4}^4&\le C\big(\|\r\dot u\|_{L^2}+\|H\cdot\na H\|_{L^2}\big)^3\big(\|\na u\|_{L^2}+\|P-\bar P\|_{L^2}+\|H\|_{L^4}^2\big)\nonumber\\
&\quad+C\big(\|\na u\|_{L^2}+\|P-\bar P\|_{L^4}+\||H|^2\|_{L^4}\big)^4\nonumber\\
&\le C\|\sqrt\r\dot u\|_{L^2}^4+\|H\|_{L^3}^4\|\na H\|_{L^6}^4+C\|\vr\|_{L^2}^2+C\|H\|_{L^3}^4\|H\|_{L^6}^4+C\|\na u\|_{L^2}^4
+C\|H\|_{L^\infty}^4\|H\|_{L^4}^4\nonumber\\
&\le C\big(\|\sqrt\r\dot u\|_{L^2}^4+\|\curl^2H\|_{L^2}^4+\|\na u\|_{L^2}^4+\|\na H\|_{L^2}^4+\|\rho-\bar\rho_0\|_{L^2}^2\big)
\nonumber\\
&\quad+C\|\curl^2H\|_{L^2}^2\|\na H\|_{L^2}^4,
\\ \label{w57}
\|\curl^2H\|_{L^2}&\le C\big(\|H_t\|_{L^2}+\|H\|_{L^\infty}\|\na u\|_{L^2}+\|\na H\|_{L^3}\|u\|_{L^6}\big)\nonumber\\
&\le C\big(\|H_t\|_{L^2}^2+\|\na H\|_{L^2}^\f12\|\curl^2H\|_{L^2}^\f12\|\na u\|_{L^2}
+\|\nabla H\|_{L^2}\|\na u\|_{L^2}\big)\nonumber\\
&\quad+C\|\na H\|_{L^2}^\f12\|\curl^2H\|_{L^2}^\f12\|\nabla u\|_{L^2}\nonumber\\
&\le \f12\|\curl^2H\|_{L^2}+C\big(\|H_t\|_{L^2}^2+\|\nabla H\|_{L^2}\|\na u\|_{L^2}^2+\|\nabla H\|_{L^2}\|\na u\|_{L^2}\big),
\end{align}
due to Lemma \ref{l24}, \eqref{f2}, and $\eqref{a1}_3$.

It follows from \eqref{f7} and \eqref{z2.5} that
\begin{align}\label{w59}
&\int_{\partial\Omega}(u\cdot\nabla n\cdot u)FdS\le C\||u|^2|F|\|_{W^{1,1}}\le C\|\nabla u\|_{L^2}^2\|F\|_{H^1}\nonumber\\
&\le \frac{1}{2}\|\sqrt{\rho}\dot{u}\|_{L^2}^2+\tilde\delta\|{\rm curl}^2H\|_{L^2}^2
+C\big(\|\nabla u\|_{L^2}^4+\|\nabla H\|_{L^2}^4\big)
+C\big(\|\nabla u\|_{L^2}^2+\|\nabla H\|_{L^2}^2\big)\nonumber\\
&\le \f12(\|\sqrt\r\dot u\|_{L^2}^2+\|H_t\|_{L^2}^2)+C\big(\|\nabla u\|_{L^2}^2+\|\nabla H\|_{L^2}^2\big).
\end{align}
%which together with \eqref{w55} and Gronwall's inequality yields that
%\begin{align*}
%\frac{d}{dt}\lb\|\sqrt\r\dot u\|_{L^2}^2+\|H_t\|_{L^2}^2\rb+\|\na\dot u\|_{L^2}^2+\|{\rm curl}\, H_t\|_{L^2}^2\le C.
%\end{align*}
%This along with \eqref{w55} indicates that
%\begin{align}
%&\frac{d}{dt}\lb\|\sqrt\r\dot u\|_{L^2}^2+\|H_t\|_{L^2}^2\rb+\|\na\dot u\|_{L^2}^2+\|{\rm curl}\, H_t\|_{L^2}^2\nonumber\\
%&\le C\big(\|\sqrt\rho\dot u\|_{L^2}^2+\|H_t\|_{L^2}^2+\|\na u\|_{L^2}^2+\|\na H\|_{L^2}^2+\|\rho-\bar\rho_0\|_{L^2}^2\big).
%\end{align}
Choose a positive constant $D_4$ and define a temporal energy functional
\begin{align*}
\cm_3(t)\triangleq D_5\cm_2(t)+\|\sqrt{\rho}\dot{u}(\cdot,t)\|_{L^2}^2+\|H_t(\cdot,t)\|_{L^2}^2
++\int_{\pa\Om}(u\cdot\na n\cdot u)FdS
+D_4\|(\rho-\bar\rho_0)(\cdot,t)\|_{L^2}^2.
\end{align*}
Note that $\cm_3(t)$ is equivalent to $\|(\rho-\bar\rho_0, \nabla u,
\nabla H, \curl^2H, \sqrt{\rho}\dot u, H_t)\|_{L^2}^2$ provided $D_4$ and $D_5$ are chosen large enough. In light of \eqref{b37}, we get that, for any $t\ge 0$,
\begin{align}\label{w58}
\frac{d}{dt}\cm_3(t)+\frac{\cm_3(t)}{D_5}+\frac{\mu\|\nabla\dot u\|_{L^2}^2+\|\nabla H_t\|_{L^2}^2}{D_5}\le 0.
\end{align}
Integrating \eqref{w58} with respect to $t$ over $[0, \infty)$ together with \eqref{w59} leads to \eqref{w42}.
\end{proof}

\begin{lemma}\label{l34}
Under the conditions of Theorem \ref{thm1}, there exist two positive constants $C_6$ and $\eta_6$, which are dependent on
$K$ and $M$, but independent of $t$, such that for any $t\ge 0$,
\begin{align}\label{w61}
\|(\rho-\bar\rho_0)(\cdot, t)\|_{L^\infty}\le C_6e^{-\eta_6t}.
\end{align}
\end{lemma}
\begin{proof}
It follows from Sobolev's inequality, Lemma \ref{l24}, \eqref{w42}, \eqref{w1}, and \eqref{w17} that
\begin{align}\label{3.53}
\|F\|_{L^\infty}^2+\|H\|_{L^\infty}^4&\le C\big(\|F\|_{H^1}\|\nabla F\|_{L^6}+\|F\|_{L^2}^2+\|H\|_{L^\infty}^4\big)\nonumber\\
&\le C\big(\|\nabla u\|_{L^2}+\|\na H\|_{L^2}+\|\curl^2H\|_{L^2}+\|\sqrt\r\dot u\|_{L^2}+\|\rho-\bar\rho_0\|_{L^2}\big)\nonumber\\
&\quad\times\big(\|\nabla\dot u\|_{L^2}+\|\curl^2H\|_{L^2}^2+\|\nabla u\|_{L^2}^2+\|\na H\|_{L^2}^2\big)\nonumber\\
&\quad+C\big(\|\nabla u\|_{L^2}^2+\|\na H\|_{L^2}^2+\|\curl^2H\|_{L^2}^4+\|\na H\|_{L^2}^4+\|\rho-\bar\rho_0\|_{L^2}^2\big)\nonumber\\
&\le C\big(\|\nabla u\|_{L^2}+\|\na H\|_{L^2}+\|\curl^2H\|_{L^2}+\|\sqrt\r\dot u\|_{L^2}+\|\rho-\bar\rho_0\|_{L^2}\big)\|\na\dot u\|_{L^2}\nonumber\\
&\quad+C\big(\|\nabla u\|_{L^2}^2+\|\na H\|_{L^2}^2+\|\curl^2H\|_{L^2}^2+\|\sqrt\r\dot u\|_{L^2}^2+\|\rho-\bar\rho_0\|_{L^2}^2\big)\nonumber\\
&\quad+C\big(\|\nabla u\|_{L^2}^4+\|\na H\|_{L^2}^4+\|\curl^2H\|_{L^2}^4\big)\nonumber\\
&\le C(1+\|\na\dot u\|_{L^2})e^{-\eta_1t},
\end{align}
which along with \eqref{w41} implies that, for any $\delta>0$, there exists a positive constant $T_1$ such that
\begin{equation}\label{3.51}\int_{T_1}^\infty
\Big(\|F\|_{L^\infty}+\frac{1}{2}\||H|^2\|_{L^\infty}\Big)dt\le \delta.
\end{equation}

Let $X(t, y)$ be the particle path given by
\begin{align}\label{lz1}
\begin{cases}
\f{d}{dt}X(t, y)=u(t, X(t, y)),\\
X(0, y)=y,
\end{cases}
\end{align}
then we rewrite the mass conservation equation \eqref{a1}$_1$ as
\begin{align*}
\frac{\mathrm{d}}{\mathrm{d}t}\ln \rho=-\divv  u.
\end{align*}
By the definition of $F$ in \eqref{f5}, \eqref{a9}, and \eqref{3.53}, we have
\begin{align*}
\rho(t, x)\ge \rho_0\text{e}^{{-\int_0^t}\|\divv  u\|_{L^\infty}\mathrm{d}\tau}\ge \rho_*\text{e}^{{-\int_0^t}C\|(P-\bar P,F,H^2)\|_{L^\infty}\mathrm{d}\tau}\ge \rho_*\text{e}^{-C(t+1)}.
\end{align*}
Motivated by \cite{LWZ25}, we claim that there exists a positive constant $c_0$ such that
\begin{equation}
\rho(t, x)\ge c_0.\label{3.50}
\end{equation}
Indeed, if \eqref{3.50} is false, there exists a time $T_2\ge T_1$ such
that $0<c_0=\rho({ x}(T_2),T_2)\le \frac{\bar\rho_0}{2^\frac{1}{\gamma}}$, then
choose a minimal value of $T_3>T_2$ such that $\rho({ x}(T_3),T_3)=\frac{c_0}{2}$.
Thus, $\rho({ x}(t),t)\in[\frac{c_0}{2},c_0]$ for $t\in[T_2,T_3]$.
Denote by $\vr\triangleq\rho-\bar\rho_0$, then $\eqref{a1}_1$ can be rewritten as
\begin{align}\label{3.52}
\f{d}{dt}\vr(t, X(t, y))+\frac{1}{2\mu+\lambda}\big(\rho(\r^\gamma-\bar\rho_0^\gamma)\big)(t, X(t, y))=
-\frac{1}{2\mu+\lambda}\bigg[\r\Big(F+\frac12|H|^2\Big)\bigg](t, X(t, y)).
\end{align}
Integrating \eqref{3.52} along particle
trajectories from  $T_2$ to $T_3$,
abbreviating $\rho(t, X(t, y))$ by $\rho(t)$ for convenience, we get that
\begin{equation}\label{b3.53}
(2\mu+\lambda)(\rho-\bar\rho_0)\Big{|}^{T_3}_{T_2}+\int^{T_3}_{T_2}
\rho(\rho^\gamma-\bar\rho_0^\gamma)\mathrm{d}\tau=-\int^{T_3}_{T_2}\bigg[\r\Big(F+\frac12|H|^2\Big)\bigg](t, X(\tau, y))\mathrm{d}\tau.
\end{equation}
We deal with the terms on the left-hand side of \eqref{b3.53} as follows
\begin{align}
(2\mu+\lambda)(\rho-\bar\rho_0)\Big{|}^{T_3}_{T_2}+\int^{T_3}_{T_2}\rho(\rho^\gamma
-\bar\rho_0^\gamma)\mathrm{d}\tau
&\le -\frac{c_0(2\mu+\lambda)}{2}+\int^{T_3}_{T_2}\frac{c_0}{2}
\bigg(\frac{\bar\rho_0^\gamma}{2}-\bar\rho_0^\gamma\bigg)\mathrm{d}\tau\nonumber
\\&= -\frac{c_0(2\mu+\lambda)}{2}-\frac{c_0}{4}(T_3-T_2)\bar\rho_0^\gamma.\label{3.54}
\end{align}
In virtue of \eqref{3.51}, we estimate the term on the right-hand side of \eqref{b3.53} as follows
\begin{align}\label{3.55}
-\int^{T_3}_{T_2}\bigg[\r\Big(F+\frac12|H|^2\Big)\bigg](t, X(\tau, y))\mathrm{d}\tau\ge -c_0\delta.
\end{align}
Substituting \eqref{3.54} and \eqref{3.55} into \eqref{b3.53}, we obtain that
\begin{align*}
-\frac{c_0(2\mu+\lambda)}{2}-\frac{c_0\bar\rho_0^\gamma}{4}(T_3-T_2)\ge -c_0\delta,
\end{align*}
which is impossible provided $\delta$ is small enough.

%Denote by $\vr\triangleq\rho-\bar\rho_0$, then $\eqref{a1}_1$ can be rewritten as
%\begin{align}\label{w62}
%\pa_t\vr+u\cdot\na\vr+\frac{1}{2\mu+\lambda}\rho(\r^\gamma-1)
%\vr=-\frac{1}{2\mu+\lambda}\Big(F+\frac12|H|^2\Big).
%\end{align}
%which along with \eqref{lz1} yields that
%\begin{align}\label{w64}
%\f{d}{dt}\vr(t, X(t, y))+\frac{1}{2\mu+\lambda}\big(\rho(\r^\gamma-1)\big)(t, X(t, y))=
%-\frac{1}{2\mu+\lambda}\bigg[\r\Big(F+\frac12|H|^2\Big)\bigg](t, X(t, y)).
%\end{align}
Multiplying \eqref{3.52} by $\vr$ and using \eqref{3.50} and \eqref{w9}, we obtain that there exist two positive constants $\eta_7$ and $C$ such that
\begin{align}\label{w65}
\f{d}{dt}\vr^2(t, X(t, y))+\eta_7\vr^2(t, X(t, y))\le C\big(\|F\|_{L^\infty}^2+\|H\|_{L^\infty}^4\big),
\end{align}
which together with \eqref{3.53} implies that
\begin{align}
\f{d}{dt}\vr^2(t, X(t, y))+\eta_7\vr^2(t, X(t, y))\le C(1+\|\na\dot u\|_{L^2})e^{-\eta_1t}.
\end{align}
Hence, one has that
\begin{align}\label{w68}
\frac{d}{dt}(e^{\eta_7t}\vr(t, X(t, y))^2)\le Ce^{\eta_7t}\big(1+\|\nabla\dot u\|_{L^2}\big)e^{-\eta_1t}.
\end{align}
Integrating \eqref{w68} along particle trajectories from $0$ to $t$, we derive from H\"older's inequality and \eqref{w41} that
\begin{align*}
\|\rho-\bar\rho_0\|_{L^\infty}^2&\le Ce^{-\eta_7t}+C\int_0^te^{-\eta_7(t-\tau)}(\|\nabla\dot u\|_{L^2}+1)e^{-\eta_1\tau}d\tau\nonumber\\
&\le Ce^{-\eta_7t}+C\int_0^\frac t2e^{-\eta_7(t-\tau)}(\|\nabla\dot u\|_{L^2}+1)e^{-\eta_1\tau}d\tau
+\int_{\frac t2}^te^{-\eta_7(t-\tau)}(\|\nabla\dot u\|_{L^2}+1)e^{-\eta_1\tau}d\tau\nonumber\\
&\le Ce^{-\eta_7t}+Ce^{-\frac{\eta_7t}{2}}\bigg(\int_0^\frac t2\|\nabla\dot u\|_{L^2}^2d\tau\bigg)^\frac12
\bigg(\int_0^\frac t2e^{-2\eta_1 \tau}d\tau\bigg)^\frac12
+Ce^{-\frac{\eta_7t}{2}}\int_0^\frac t2 e^{-\eta_1\tau}d\tau
\nonumber\\
&\quad+Ce^{-\frac{\eta_1t}{2}}\int_{\frac t2}^te^{-\eta_7(t-\tau)}d\tau
+Ce^{-\frac{\eta_1t}{2}}\bigg(\int_{\frac t2}^t\|\nabla\dot u\|_{L^2}^2d\tau\bigg)^\frac12
\bigg(\int_0^\frac t2e^{-2\eta_7(t-\tau)}d\tau\bigg)^\frac12\nonumber\\
&\le C\Big(e^{-\frac{\eta_7t}{2}}+e^{-\frac{\eta_1t}{2}}\Big),
\end{align*}
from which, the conclusion follows.
\end{proof}

Now we are ready to prove Theorem \ref{thm1}.
\begin{proof}[Proof of Theorem \ref{thm1}]
On one hand, \eqref{1.8} follows from \eqref{w1}, \eqref{w17}, and \eqref{w42}. On the other hand, by \eqref{w17}, \eqref{w41}, \eqref{w57},
and Gagliardo--Nirenberg inequality, we find that
\begin{align*}
\|H\|_{L^\infty}\le C\|\na H\|_{L^2}^\f12\|\curl^2H\|_{L^2}^\f12+\|\na H\|_{L^2}\le Ce^{-\eta_1t}.
\end{align*}
This combined with \eqref{w61} yields \eqref{b1.11}.
%\begin{align}
%\|(\rho-\bar\rho_0, H)\|_{L^\infty}\le Ce^{-\eta_2t},
%\end{align}
\end{proof}

\section*{Conflict of interest}
The authors have no conflicts to disclose.

\section*{Data availability}
No data was used for the research described in the article.

%\section*{Authors contributions}
%The authors claim that the research was realized in collaboration with the same responsibility. All authors read and approved the last version of the manuscript.

\end{document}